\documentclass[11pt]{amsart}
\usepackage{amsmath,amssymb,amsmath,amsthm}
\usepackage{appendix}
\usepackage{tikz}
\usepackage{tikz-cd}
\usepackage{pdflscape}
\newtheorem{lemma}{Lemma}[section]
\newtheorem{defn}[lemma]{Definition}
\newtheorem{thm}[lemma]{Theorem}
\newtheorem{prop}[lemma]{Proposition}
\newtheorem{corr}[lemma]{Corollary}

\newtheorem{rmk}[lemma]{Remark}
\usepackage[T1]{fontenc}
\usepackage{hyperref}
\usepackage{mathrsfs} 
\usepackage{glossaries}
\newacronym{dga}{DGA}{Differential Graded Algebra}
\usepackage{tensor}
\usepackage{enumerate}
\usepackage{slashed}
\usepackage{mathtools}
\usepackage{color}

\usepackage{changes}
\definechangesauthor[name={Yang Liu}, color=orange]{YL}

\usepackage[inline]{showlabels}

\showlabels[\small\color{gray}]{bibitem}

\let\SS\S 
\newcommand{\abs}[1]{\left\vert#1\right\vert}
\newcommand{\set}[1]{\left\{#1\right\}}

\newcommand{\brac}[1]{\ensuremath{\left( {#1} \right)}}
\newcommand{\abrac}[1]{\ensuremath{\left\langle {#1}
\right\rangle}}
\newcommand{\sbrac}[1]{\ensuremath{\left[ {#1} \right]}}
\newcommand{\lrabrac}[3]{\ensuremath{\tensor[_{#1}]{\abrac{#2}}{_{#3}}}}

\DeclareMathOperator{\Tr}{Tr}

\DeclareMathOperator{\Res}{Res}
\DeclareMathOperator{\diag}{Diag}
\newcommand{\du}{\delta} 
\newcommand{\nres}{\mathscr W} 

\newcommand{\defeq}{:=}

\def\Z{\mathbb Z}

\def\be{\begin{equation}}
\def\ee{\end{equation}}

\foreach \x in {A,...,Z}{\expandafter\xdef\csname \x\endcsname{\noexpand\mathcal{\x}}} 
\foreach \x in {A,...,Z}{\expandafter\xdef\csname I\x\endcsname{\noexpand\mathbb{\x}}}
\foreach \x in {A,...,Z}{\expandafter\xdef\csname frak\x\endcsname{\noexpand\mathfrak{\x}}}
\foreach \x in {a,...,z}{\expandafter\xdef\csname frak\x\endcsname{\noexpand\mathfrak{\x}}}

\usepackage{mdframed}

\title{  Algebraic Versus Spectral Torsion}

\author[L. Dąbrowski]{Ludwik Dąbrowski}
\address[L. Dąbrowski]{SISSA, via Bonomea 265, 34136, Trieste}
\email{dabrow@sissa.it}

\author[Y. Liu]{Yang Liu}
\address[Y. Liu]{SISSA, via Bonomea 265, 34136, Trieste}
\email{yliu@sissa.it}

\author[S. Mukhopadhyay]{Sugato  Mukhopadhyay}
\address[S.  Mukhopadhyay]{SISSA, via Bonomea 265, 34136, Trieste}
\email{smukhopa@sissa.it}

\begin{document}
\maketitle

Abstract. 
We relate the recently defined spectral torsion with 
the algebraic torsion of noncommutative differential calculi
on the example of the almost-commutative geometry of the product of a closed oriented Riemannian spin manifold $M$ with the two-point space $\Z_2$.

\tableofcontents

\section{Introduction}
Let $D_T$ be a Dirac operator coupled to the torsion tensor $T$ on a closed oriented Riemannian spin manifold $M$. 
A {\em torsion functional} $\mathcal{T}$ of three differential 1-forms on $M$ was recently introduced in \cite{DSZ24} in terms of Wodzicki residue. 
It permits to recover the torsion $T$,
and being of spectral nature generalizes to noncommutative geometry using the results of [CM].
Indeed for any Dirac operator of a finitely summable regular spectral triple one has an analogous functional $\mathcal{T}$,
from which a {\it quantum} analogue of the torsion $T$ can be read off.\\
Obviously, various notions of torsion and the related Levi-Civita connection have been extensively studied until now on the algebraic (polynomial) level of noncommutative differential calculi, see eg. \cite{BG20} and the references therein.   
In this paper we initiate analysis how these two approaches are related one to another.

They are a priori quite different. The spectral one is intrinsic to
the given spectral triple, that is, no more input required. In the algebraic
approach, instead of relying on the quantum analogues of Dirac or Laplace operators,
the torsion is defined with respect to the choices of a differential calculus (at
least of order two) and  a connection on one-forms. 
Therefore one has to engage some common territory, which we take as the algebraic first order differential calculus realized in terms of operators associated to the spectral triple, and then construct a suitable second order differential calculi.

As the study case we analyse the example of almost commutative geometry on the product of a closed oriented Riemannian spin manifold $M$ (for simplicity of even dimension) with the two-point space $\Z_2$, aka. Connes-Lott model.
It has been shown in \cite{DSZ24} to have nonvanishing spectral torsion
functional $\mathcal{T}$ and so the quantum torsion tensor $T$.
This torsion $T$ clearly originates from the factor $\Z_2$,
as on the first factor $M$ the canonical Dirac operator associated to the Levi-Civita connection is used. 
Though the spectral torsion on $\Z_2$ could not be immediately captured by the method of \cite{DSZ24}, in this paper we accomplish it by using matrix trace as a natural extension of Wodzicki residue to finite spectral triples. On the algebraic side, 
we first convey the algebraic torsion as a linear map from 1-forms to 2-forms
into a trilinear functional of 1-forms, see Definition \ref{defn:T}. 
Our first main result Theorem \ref{z2agree} shows that  there exists  a unique
linear connection that gives the exact match.   

However the accordance of the spectral and algebraic approaches for the full almost-commutative geometry on $M\times \Z_2$ turns out to be more involved than the aforementioned steps for $\Z_2$.
We first work with Connes' differential calculus 
\footnote{developed in Prop. 4 and 6 in \cite[Ch. 6,\SS1]{Con-94}.},
and observe that in the setting of \cite[\SS4.2]{DSZ24}, 
so called {\em junk forms} on the manifold kill the torsion generated from the
two-point space, cf. the end of \SS \ref{sec:AT-C} for detailed discussion. 
Therefore only a partial agreement with the spectral side can be achieved, namely,
one has to adjust the spectral functional by the projection $\sigma_2$ in
Connes' calculus.  The precise statement is recorded in Theorem \ref{thm:TvsCn}.

To overcome this problem we adopt a recent modification 
in \cite{MR24} of the algebraic approach and provide its ingredients appropriate for our spectral triple of $M\times \Z_2$.
Next, since the latter one is a product type, which corresponds essentially to the metric product of $M$ with $\Z_2$, we use a product-type connection, which however realizes only part of the spectral torsion of \cite{DSZ24}. 
In order to overcome this unexpected {\em impasse} we have resort to a connection of non-product type by adding a suitable mixing perturbation term. With this we ultimately achieve a perfect agreement with the spectral approach, which constitutes our second main result formulated as Theorem \ref{2nd main}.
It could be also mentioned that the example of the almost-commutative geometry on $M\times \Z_2$ provides an interesting 
and non-trivial instance of the approach in \cite{MR24}.

\bigskip
\noindent
\subsubsection*{Acknowledgements}
This research is part of the EU Staff Exchange project 101086394 ``Operator
Algebras That One Can See''. It was partially supported by the University of
Warsaw Thematic Research Programme  ``Quantum Symmetries''.

\section{Algebraic and Analytic Preliminaries}

\subsection{Spectral Triples}
Spectral triples are templates of
geometric spaces in noncommutative geometry.
\begin{defn}
A spectral triple $ \brac{ \mathcal A , \mathcal H, D }$ is given by 
a unital $*$-algebra  $ \mathcal A$ with a faithful representation 
$\pi: \mathcal A \to B(\mathcal H)$ 
on the Hilbert space $\mathcal H$, and 
$D$ is a densely defined self-adjoint operator on $\mathcal H$ with compact resolvent 
and bounded commutators $[D ,a]$ with  $a \in \mathcal A$.
Furthermorem, a spectral  triple is called even if there exits a grading operator
$\gamma$ on $ \mathcal H$: $ \gamma^2 = 1$, $\gamma = \gamma^*$,
which commutes with $a \in \mathcal A$ and
anti-commutes with $D$.
\end{defn}
Hereafter we assume that $\brac{ \mathcal A , \mathcal H, D }$ is
\begin{itemize}
\item 
$n$-summable for some $n > 0$, i.e 
the eigenvalues of $ \abs D$ asymptotically grow as 
$\mu_l = O( n^{-l})$;
\item {\em regular}, i.e. the map 
\begin{align}
\label{eq:regT}
 t\mapsto \exp\brac{it\abs D} T \exp\brac{-it\abs D}
\end{align}
is smooth for $T\in \mathcal A\cup [D, \mathcal A]$.

\end{itemize}

Let $ \mathcal O \subset B (\mathcal H)$ 
be the algebra generated by $a$,  $[D,a]$ 
for $a \in \mathcal A$, and their images under 
iterated actions of the commutator $[ \abs D , \cdot]$
(cf. Definition 1.132  \cite{connes2008noncommutative}). 
For $b \in \mathcal{O}$, the spectral zeta functions $ \zeta_b (z) = \Tr ( b \abs D^{-z})$
are analytic on the half-plane $ \Re z > n$, 
where $n$ is the summability of  $D$,
and admits meromorphic continuation to the whole $\mathbb{C}$.
In this paper, we always assume that only simple poles occur. 

The residue at zero: 
\begin{align}
\label{eq:nres}
\nres ( Q)
:=
\Res_{ s=0}\Tr \brac{ 
Q \abs D^{-s}
}
,
\end{align}
defines a tracial functional on the algebra of pseudodifferential operators 
generated by $ \mathcal A$ and $[ D , \mathcal A]$ and
$ \abs D^{z}$ with $z \in \mathbb{C}$.
It extends Wodzicki residue, originally defined for pseudodifferential
operators on manifolds, to the general spectral triple framework.
Another crucial property is that the residue functional also computes the Dixmier trace $\Tr^+$. 
In more detail, such operators $Q$ of order $-m$ are measurable elements of the Schatten ideal $\mathcal L^{(1,\infty)} $
and the two functionals are proportional, 
cf.\,\cite[Prop.\,II.1\,and\,Appendix\,A]{MR1334867} 
\footnote{ In comparison with notations,  
the functional in \eqref{eq:nres} is $1/2$ of the functional
$\tau_0$ in \cite[ Prop.\,II.1]{MR1334867} }
\begin{align}
\label{eq:prptoDixtr}
\nres (Q)  \propto \Tr^+ ( Q ) .
\end{align}

The summability recovers the dimension of the manifold  
for the spin spectral triples
$ ( C^\infty(M), L^2( \slashed S) , \slashed D)$,
and as the analogue of the volume functional, we take
\begin{align}
\label{eq:wint}
\int^\nres  Q : = \nres \brac{ Q \abs D^{-n} } 
\end{align}
for any pseudodifferential operator $Q$.
When restricted to the algebra $ \int^\nres : \mathcal A \to \mathbb{C}$, 
it extends the integration of functions on manifolds 
$ \int_M ( \cdot ) \mathrm d \mathrm{vol} : C^\infty(M) \to \mathbb{C}$
to the spectral triples  setting. 

Later, to construct the inner product \eqref{eq:InnPrd}, 
we also need the positivity of $\int^{\nres}$,
inherited from the Dixmier trace \eqref{eq:prptoDixtr},
viewed as a linear functional on the algebra all zero-order pseudodifferential
operators.

\begin{defn}[ \added{\cite{DSZ24}}]
\label{defn:T-DSZ}
Let $ \brac{ \mathcal A , \mathcal H, D }$ be a $n$-summable regular spectral 
triple, the spectral torsion functional is the following $\mathbb{C}$-trilinear functional 
on one-forms$:$
\begin{align}
\label{eq:T-DSZ}
\mathscr T_D ( u ,v , w ) := \int^\nres u v w D , 
\, \,
u,v,w \in \Omega^1_D ( \mathcal A).
\end{align}
Moreover, if $ \mathcal H$ is finite dimensional, \eqref{eq:T-DSZ} is simply 
replaced by the trace of matrices:
\begin{align*}
\mathscr T_D ( u ,v , w ) := \Tr \brac{  u v w D } , 
\, \,
u,v,w \in \Omega^1_D ( \mathcal A).
\end{align*}
\end{defn}

We recall now some basic constructions needed later when working with Hermitian structures.
\begin{defn}
\label{defn:A-iprd}
Given a pre-C$^*$-algebra $\mathcal A$ a left pre-Hilbert $\mathcal A$-module is  
a left  $\mathcal A$-module  $\mathcal E$, 
together with  an $ \mathcal A$-valued inner product 
$\lrabrac{\mathcal A}{ \cdot , \cdot}{}$ that satisfies:
\begin{itemize}
\item
$\lrabrac{\mathcal A}{ \cdot , \cdot}{}$ is $\mathbb{C}$-linear in the first
argument,
\item
$  
\lrabrac{\mathcal A}{ a \cdot x , y}{}
= 
a \cdot \lrabrac{\mathcal A}{  x , y}{}
$,
for $ x, y \in \mathcal E$ and $ a \in \mathcal A$. 
\item 
$ \lrabrac{\mathcal A}{  x , y}{}^* = \lrabrac{A}{  y ,x   }{}$,    

\item
$ \lrabrac{\mathcal A}{  x , y}{} \ge 0$
for all $x \in \mathcal E$, 
and the equality only holds for $x = 0$.
\end{itemize}
\end{defn}

The model example: $ \mathcal E =  \mathcal A$, 
\begin{align*}
\lrabrac{\mathcal A}{ x, y}{} = x y^* , \, \,
\forall x, y \in \mathcal E = \mathcal A .
\end{align*}

Let $ \mathcal A$, $ \mathcal B$ and $ \mathcal C$ be $*$-algebras. 
Given bimodules
$ \mathcal E = \tensor[_{\mathcal A}]{\mathcal E}{_{\mathcal B}}$
with $ \lrabrac{\mathcal A}{ \cdot , \cdot }{} :
\mathcal E \otimes \mathcal E \to \mathcal A$ and
$ \mathcal F  = \tensor[_{ \mathcal C}]{\mathcal F}{_{\mathcal A}}$
with $ \lrabrac{ \mathcal C}{  \cdot , \cdot }{} :
\mathcal F \otimes \mathcal F \to \mathcal C$,
the balanced tensor product
$ \mathcal F \otimes_{\mathcal A} \mathcal E$
is a pre-Hilbert $ \mathcal C$-module with 
with the $\mathcal C$-valued inner product 
\begin{align}
\label{eq:ip-corsp}
\lrabrac{\mathcal C}{ x \otimes y , u \otimes v}{}
=
\lrabrac{\mathcal C}{ x  \cdot \lrabrac{\mathcal A}{ y,v }{} , u  }{},
\end{align}
where $x ,u \in \mathcal F$ and $ y , v \in \mathcal E$.\\

\subsection{Differential  Calculi}
For a unital $*$-algebra $ \mathcal A$, the universal differential calculus 
refers to the universal \gls{dga} 
$ (\Omega_u ( \mathcal A) 
:= \bigoplus_{j=0 }^\infty \Omega^k_u  ( \mathcal A), \du)$,
which has $ \Omega^0_u ( \mathcal A) = \mathcal A$ as degree zero 
and is generated by symbols $ \du a$,
of degree one, subject to the relations $ \du (1)  =0$ and
$ \du ( a b) = a \du b +  ( \du a) b$, 
where $a , b \in \mathcal A $.

As an $ \mathcal A$-bimodule, the space of one-forms
$ \Omega^1_u ( \mathcal A)$ is isomorphic 
to the kernel of the
multiplication map $m: \mathcal A \otimes \mathcal A \to \mathcal A$ via
\begin{align*}
\sum_{  }
a_i \otimes b_i \mapsto 
\sum_{  } 
a_i \du b_i \in \Omega^1 (\mathcal A).
\end{align*}
In general, the space of universal $k$-forms is the $k$-fold 
balanced tensor product over  $\mathcal A$,
$ \Omega^k_u ( \mathcal A) 
:= \brac{ \Omega^1_u ( \mathcal A) }^{\otimes_{ \mathcal A} k}$, in traditional notation, 
consisting of finite sums
\begin{align*}
    \sum_{  } a_0 \du a_1 \cdots \du a_k, \, \,
    a_0 , \ldots , a_k \in \mathcal A.
\end{align*}
The differential is defined by
\begin{align}
\label{eq:d-ucal}
\du   
\brac{ a_0 \du a_1 \cdots \du a_n } 
=
\du a_0 \du a_1 \cdots \du a_n, 
\end{align}
satisfying $ \du^2 = 0$ and 
\begin{align*}
\du ( \omega_1 \omega_2  ) 
=
(\du  \omega_1   ) \omega_2
+ 
(-1)^{\deg \omega_1} \omega_1 ( \du \omega_2),
\, \,
\forall \omega_1 , \omega_2 \in \Omega ( \mathcal A).
\end{align*}
The multiplication of $ \Omega_u ( \mathcal A)$
is dictated by the graded Leibniz property above. 
The $*$-involution of  $ \mathcal A$ extends to $  \Omega_u (\mathcal A)$
via $ ( \du a)^* := -\du a^*$.

Now given a spectral triple $ ( \mathcal A , \mathcal H, D)$, 
the representation $ \pi : \mathcal A \to B(\mathcal H)$
extends  to universal one-forms via 
$\du \to [D , \cdot ]$, with the image $ \Omega^1_D (\mathcal A)$ called the one-forms, or a first order
differential calculus, associated with the spectral triple. Explicitly, we set 
\begin{align*}
\pi_D : \Omega^1_u ( \mathcal A) \to  \Omega^1_D ( \mathcal A) 
\subset B(\mathcal H):
a \du b \mapsto a [D , b],
\end{align*}
where $a , b \in \mathcal A$.

\begin{defn}
The tensor algebra $T_D ( \mathcal A)$ associated with the one-forms 
\begin{align*}
 T_D ( \mathcal A) := \bigoplus_{ k=0}^\infty T_D^k ( \mathcal A)
 ,\, \,
T_D^k ( \mathcal A) 
:= \brac{ \Omega^1_D ( \mathcal A) }^{\otimes_{ \mathcal A} k}
\end{align*}
is the direct sum of all $k$-fold balanced tensor products of  
$ \Omega^1_D ( \mathcal A)$.
The map  $\pi_D$ extends naturally 
 \begin{align*}
\pi^{ \otimes k}_D : \Omega^k_u ( \mathcal A) \to T_D^k ( \mathcal A):
w_1 \otimes \ldots \otimes w_k
\mapsto
\pi_D (w_1) \otimes \ldots \otimes \pi_D ( w_k).
\end{align*}
In particular, one obtains a $*$-algebra structure on  $ T_D ( \mathcal A)$
from  $ \Omega_u  ( \mathcal A)$ under such identification.
\end{defn}

Note that  $ \Omega^1_D ( \mathcal A) \subset B(\mathcal H)$, 
we can compose $ \pi_D^{\otimes k}$ with the multiplication map,
\begin{align*}
m: T_D^k ( \mathcal A) \to B ( \mathcal H),
\end{align*}
which extends the representation $\pi: \mathcal A \to B(\mathcal H)$ 
to all universal differential forms:
$\widehat{\pi}_D = \oplus_{ k=0}^\infty \widehat{ \pi}_D^{\otimes k}$
where
$ \widehat{\pi}_D^{\otimes k}: = m \circ \pi_D^{\otimes k}:
 \Omega^k_u (\mathcal A) \to B(\mathcal H)$
\begin{align}
\label{eq:pi}
\widehat{  \pi}_D^{\otimes k}
\brac{ a_0 \du a_1 \cdots \du a_n }  
= a_0 [D , a_1] \cdots [ D , a_n ]  \in B ( \mathcal H).
\end{align}

Nevertheless, the differential structure,
i.e., $\du$ in \eqref{eq:d-ucal}, does not carry over, namely,
there exists universal forms $w \in \Omega^k_u ( \mathcal A)$
such that $ \widehat{ \pi}_D (w) = 0$, but 
$ \widehat{ \pi}_D  (\du(w) ) \neq 0$.
One needs to pass to suitable quotients of $ T_D(\mathcal A)$ 
and $ \Omega_D ( \mathcal A)$ in order to obtain \gls{dga}s.
Let $J_0 (\mathcal A)$ be the kernel of $ \widehat{ \pi}_D$, which is a graded two-sided ideal with the $k$-th component:   
$J_0^{k} (\mathcal A) = \Omega^k_u ( \mathcal A) \cap J_0 (\mathcal A)$.
Then 
\begin{align}
\label{eq:JD}
J_D (\mathcal A) = 
\bigoplus_{ k=1}^\infty    J^k_D ( \mathcal A),
\, \,
J^k_D ( \mathcal A) = J_0^k ( \mathcal A)+ \du \brac{  J_0^{k-1} ( \mathcal A) } 
\end{align}
is a graded two-sided differential ideal of $ \Omega_u ( \mathcal A)$, 
known as the space of junk forms.

\begin{defn}
The quotient space
\begin{align*}
\Omega_D ( \mathcal A) =
\widehat{ \pi}_D \brac{  \Omega_u (\mathcal A) }  / \widehat{ \pi}_D \brac{ J_D }.  
\end{align*}
is  a \gls{dga}, known as Connes' calculus associated with the spectral triple,
with the differential 
\begin{align}
\label{eq:dD}
\du_D :
\Omega_D^k ( \mathcal A) \to \Omega_D^{k+1} ( \mathcal A)
, \, \,
\du_D \sbrac{  \pi (w)} 
= 
\sbrac{ \pi ( \du w )  }_{ k+1 } ,
\end{align}
for any universal $k$-form  $w \in \Omega^k_u( \mathcal A)$.
\end{defn}

In \cite{MR24}, similar notion was introduced for
$T_D (\mathcal A)$, called junk tensors. 
\begin{defn}
The bimodule of degree $k$ junk tensors  
$JT^k_D (\mathcal A) \subset T^k_D (\mathcal A)$ consists of elements of the form:
\begin{align*}
\set{
T \in T^k_D ( \mathcal A) : T = { \pi}_D ( \delta (w) ) 
,\, \,
w \in JT^{(k-1)}_D (\mathcal A) = J_D (\mathcal A) \cap \Omega^{k-1}_u ( \mathcal A)
}.
\end{align*}
\end{defn}
This makes $ T_D (\mathcal A) / JT_D (\mathcal A)$ a \gls{dga}.

In the next two subsections, 
we will present two approaches to bring in extra structures, 
especially analytic ones,
to avoid working with quotient spaces.

\subsection{ Connes' Construction }
We first consider the following  GNS-inner product associated with
the positive \footnote{the positivity inherited from  that of the Dixmier trace,
see \eqref{eq:prptoDixtr}}
linear form in \eqref{eq:wint}:
\begin{align}
\label{eq:InnPrd}
\abrac{ T_1 , T_2}_k = 
\int^\nres T_2^* T_1
, \, \,\, \,
T_1 , T_2 \in \widehat{ \pi}_D \brac{ \Omega^k_u ( \mathcal A) }
.
\end{align}

Denote by $ \mathfrak H_k $  the Hilbert space completion of
$ \widehat{ \pi}_D \brac{ \Omega^k_u ( \mathcal A) }$ with the decomposition
\begin{align}
\label{eq:sig-k}
\mathfrak H_k
=
\sigma_k \mathfrak H_k \oplus  
(1-\sigma_k)\mathfrak H_k 
\end{align}
where $ \sigma_k $ is the  projection on to the orthogonal complement of 
the subspace of junk forms
$ \widehat{ \pi}_D \brac{ \du \brac{ J_0 \cap \Omega^{k-1} (\mathcal A)} }$.
By construction, 
we have a well-defined inner product on $  \Omega^k_D ( \mathcal A) $:
\begin{align*}
\abrac{  [ T_1 ]_k,  [   T_2  ]_k} 
\defeq  \abrac{  \sigma_k (T_1 ), \sigma_k  (T_k) }_k
, \, \,\, \,
T_1 , T_2 \in \widehat{ \pi}_D \brac{ \Omega^k_u ( \mathcal A) }
.
\end{align*}

\begin{defn}
We denote by $ \Lambda^k_D (\mathcal A)$ the image of $\sigma_k$:
\begin{align*}
\sigma_k : \Omega^k_D ( \mathcal A) \to  \Lambda^k_D (\mathcal A) 
\subset \mathfrak H_k
\end{align*}
which are the analogue of $k$-forms in Connes' construction for a given
spectral triple.
\end{defn}

The \gls{dga} structure upto degree two reads
\begin{align}
\label{eq:C-cplx}
0 \to \mathcal A \xrightarrow{ \delta_D}  \Omega^1_D ( \mathcal A)
\xrightarrow{  \mathrm d_{ \sigma_2 } } \Lambda^2_D ( \mathcal A) \to \cdots ,
\end{align}
where $ \mathrm d_{ \sigma_2 } := 
\sigma_2 \circ \widehat{\pi_D}  \circ \delta \circ \widehat{ \pi}^{-1}_D  $ is given by 
\begin{align}
\label{eq:d2-sig}
\mathrm d_{ \sigma_2 } ( a [D , b] ) = 
\sigma_2 \brac{ [D ,a ] [D , b] }, 
\, \, a ,b \in \mathcal A.
\end{align}

\subsection{Mesland-Rennie construction.}
\label{sec:MR-cal}
Compared to Connes' construction above, 
 \cite{MR24} works with elements of 
$m^{-1}\Lambda_D ( \mathcal A)\subset T_D (\mathcal A)$.
We recall construction of the second order differential calculi,
which is sufficient for our discussions related to the torsion. 

For each idempotent $ \Psi : T_D^2 ( \mathcal A) \to T_D^2 ( \mathcal A)$, 
$ \Psi = \Psi^2$, 
which satisfies 
\begin{equation}\label{psi}
	JT^2_D \subseteq \mathrm{Im}(\Psi) \subseteq m^{-1}(J^2_D),
\end{equation}
one has the  second order differential calculus of $\mathcal{A}$
\begin{align}
\label{eq:Psi-cplx}
0 \to \mathcal A \xrightarrow{ \delta_D}  \Omega^1_D ( \mathcal A)
\xrightarrow{  \mathrm d_{ \Psi } } T^2_D ( \mathcal A)  ,
\end{align}
where
\begin{align*}
\mathrm d_\Psi :=  ( 1 - \Psi)
\brac{ \widehat{ \pi}_D \circ \du \circ \widehat{ \pi}_D^{-1}  }
: \Omega^1_D ( \mathcal A) \to  T_D^2 ( \mathcal A),
\end{align*}
that is:
\begin{align*}
\mathrm d_\Psi ( a[ D , b] ) :=
(1-\Psi) \brac{ 
[D , a] \otimes_{ \mathcal A} [ D , b] 
}
,\, \, \forall a, b \in \mathcal A.
\end{align*}
It is well-defined due to the left inclusion in \eqref{psi}, and satisfies
$ \mathrm d_{ \Psi } \circ \delta_D = 0$. 

Moreover, if $ \Omega^1_D ( \mathcal A)$ admits a Hermitian structure, that is,
an  $ \mathcal A$-valued inner product in the sense of Definition
\ref{defn:A-iprd}, then it induces a Hermitian structure on $ T_D^2 ( \mathcal A)$
according to \eqref{eq:ip-corsp}:
\begin{align}
\lrabrac{\mathcal A}{ x \otimes y , u \otimes v}{}
=
\lrabrac{\mathcal A}{ x  \cdot \lrabrac{\mathcal A}{ y,v }{} , u  }{},
\end{align}
together with the volume functional $ \int^\nres$ defined in \eqref{eq:wint},
we have a scalar inner product 
\begin{align*}
\abrac{  x \otimes y , u \otimes v } := 
\int^\nres
\lrabrac{\mathcal A}{ x \otimes y , u \otimes v}{} ,
\end{align*}
for all $x, y , u, v \in \Omega^1_D ( \mathcal A)$.
In this case, we further require $ \Psi$ to be a projection, namely, 
$ \Psi^2 = \Psi$ and $ \Psi = \Psi^*$ regarding the inner product above.

\subsection{ Algebraic  Torsion of Connections on $\Omega^1_D ( \mathcal A)$}
\label{sec:algebraicT}

In the algebraic setting, torsion is a notion attached to connections 
on the $\mathcal A$-bimodule $ \Omega^1_D ( \mathcal A)$.
\begin{defn}
Given a $\mathcal A$-bimodule $ \mathcal E$, and the first order calculus 
$ \Omega^1_D ( \mathcal A)$,
a (left) connection is a 
$\mathbb{C}$-linear map 
$ \nabla: \mathcal E \to \Omega^1_D ( \mathcal A) \otimes_{ \mathcal A} \mathcal E$
satisfying the Leibniz rule: 
$ \nabla ( a w) = a \otimes \nabla w 
+ [D ,a ] \otimes w$, where $a \in \mathcal A$ and $w \in \mathcal E$. 
\end{defn}

We are interested in the case $ \mathcal E = \Omega^1_D ( \mathcal A)$ so $\nabla : 
 \Omega^1_D ( \mathcal A)\longrightarrow T^2_D (\mathcal A)$.
\\
We will often use the Sweedler type notation 
\begin{align}
\label{eq:SwNtn-L}
\nabla (w) = \brac{ \nabla w }_{ (1)} \otimes \brac{ \nabla w }_{ (0)}
,\, \,
\end{align}
where the subscript $(1)$ indicates the one-form factor generated by the
connection and  $(0)$ stands for the factor of the module  $ \mathcal E$.
Of course,  
the difference of two (left)  connections is a left 
$\mathcal A$-module map because of the Leibniz rule.

The notion of torsion measures the difference 
between a connection and the differential at degree two in the related differential calculus.
\begin{defn} \label{defn:tor}
Let $ \nabla: \Omega^1_D (\mathcal A) \to T_D^2 ( \mathcal A)$ 
be a  connection.
The torsion of $ \nabla $ is defined as follows$:$
\begin{enumerate}[$1)$]
\item
Regarding the differential calculus in \eqref{eq:C-cplx},
\begin{align}
\label{eq:TorL}
T_\sigma ^{\nabla} :=  \sigma_2 \circ m \circ \nabla - \mathrm d_{ \sigma_2 } : 
\Omega^1_D ( \mathcal A)  \to \Lambda^2_D ( \mathcal A) .
\end{align}

\item With respect to \eqref{eq:Psi-cplx}, 
\begin{align}
\label{eq:TorL-Y}
T_\Psi ^{\nabla} :=  (1-\Psi) \circ \nabla -  \mathrm d_{\Psi} : 
\Omega^1_D ( \mathcal A)  \to T^2_D ( \mathcal A) .
\end{align}
\end{enumerate}
\end{defn}
In particular,  they are both determined by the evaluations on the image of 
$ \delta_D: \mathcal A \to \Omega^1_D (\mathcal A)$, that is, for $a \in \mathcal A$,
\begin{align}
\label{eq:TorSwd}
T^{\nabla}_{ \sigma_2 } ( [D , a] ) 
&= 
\sigma_2 \brac{ \nabla( [D , a])_{ (1)}   \nabla( [D , a])_{ (0)}
},
\\
T^{\nabla}_{ \Psi } ( [D , a] ) 
&= 
\brac{ 1-\Psi }
\brac{ \nabla( [D , a])_{ (1)} \otimes_{ \mathcal A}  \nabla( [D , a])_{ (0)} }
.
\label{eq:TorSwd-Y}
\end{align}

Because of the left Leibniz property, 
the difference of two  connections $ \nabla $ and $ \widetilde \nabla$
\begin{align*}
S = \widetilde \nabla - \nabla :
\Omega^1_D ( \mathcal A) \to
T^2_D (\mathcal A)
\end{align*}
is a left $ \mathcal A$-module map. 

The difference of their algebraic torsion maps can be seen as follows:
\begin{prop}
\label{prop:tor-purt}
Keep the notations as above, and set $ S_\sigma := \sigma_2 \circ m \circ S:
\Omega^1_D ( \mathcal A) \to \Lambda^2_D (\mathcal A) $  and 
$ S_\Psi := (1- \Psi) \circ S : \Omega^1_D ( \mathcal A) \to 
 T^2_D ( \mathcal A)$, then 
\begin{align}
\label{eq:tor-purt}
T_\sigma^{ \widetilde \nabla} = 
T_\sigma^{ \nabla} + S_\sigma
,\, \,  
T_\Psi^{ \widetilde \nabla} = 
T_\Psi^{ \nabla} + S_\Psi 
.
\end{align}
\end{prop}
\begin{proof}
Forthe first equation in \eqref{eq:tor-purt} following the definition \eqref{eq:TorL}, we have
\begin{align*}
T_\sigma^{ \widetilde \nabla}
& = 
\sigma_2 \circ m \circ \brac{  \nabla + S}  - \mathrm d_{ \sigma_2 }
=
\sigma_2 \circ m \circ \nabla  - \mathrm d_{ \sigma_2 } +
\sigma_2 \circ m \circ S
\\
& =
T_\sigma^{ \nabla}  +  
S_\sigma ,
\end{align*}
The calculations for the second equation are similar.
\end{proof}

To make comparison with the spectral torsion functional introduced in 
\cite{DSZ24}, cf. Definition \ref{defn:T-DSZ}, 
we consider the following associated trilinear functionals.
\begin{defn} \label{defn:T}
For a given (left) connection $ \nabla$ on $ \Omega^1_D ( \mathcal A)$, 
we set 
\begin{align}
\label{eq:Tsig}
\mathscr T_{ \sigma_2 } ( u , v, w) 
&=
\int^\nres u v T^{\nabla}_{ \sigma_2} (w)
,
\\
\label{eq:TPsi}
\mathscr T_{ \Psi} ( u , v, w) 
&=
\int^\nres u v m \brac{T^{\nabla}_{ \Psi } (w)} ,
\end{align}
where $ u, v, w \in \Omega^1_D ( \mathcal A) \subset B(\mathcal H)$
are viewed as bounded operators, and
$m: T^2_D (\mathcal A) \to B(\mathcal H)$ is the multiplication map
so that $  m \brac{T^{\nabla}_{ \Psi} (w)} \in B ( \mathcal H)$.

Moreover for spectral triple with the Hilbert space of finite dimension,  as in
Def.\,\ref{defn:T-DSZ} we simply use the matrix trace $\Tr$, instead of
$\int^\nres$, to  define the functionals.
\end{defn}

\section{ 
The Almost Commutative Geometry of  $M \times \Z_2 $}
We start with the quantum geometry of the two-point space $\Z_2$,
which is the building block for $M \times \Z_2 $.

\subsection{The two-point space $\Z_2$}
\label{sec:2pt}
We use notation of \cite[Ch.6, \SS3]{Con-94}. 
The algebra $ \mathcal A_{\Z_2} $ of the spectral triple is the 
space of (complex) functions  on the two-point space $\Z_2 = \set{+ , -}$, that is just the direct sum $ \mathbb{C} \oplus \mathbb{C}$. 
The Hilbert space and the Dirac operator read
$ \brac{ \mathfrak h , D_\phi }$: 
\begin{align}
\label{eq:kcyc0}
\mathfrak h = \mathfrak h_+ \oplus \mathfrak h_- 
, \, \,
D_\phi = 
\begin{bmatrix} 
0 & \phi   \\
\phi^* & 0
\end{bmatrix} 
, \, \,
\end{align}
where $ \mathfrak h_{ \pm}$ are  finite dimensional Hilbert spaces
\footnote{
the dimensions of $ \mathfrak h_+$ and $ \mathfrak h_-$ can be different.}
and
$\phi: \mathfrak h_+ \to \mathfrak h_-$ is $\mathbb{C}$-linear with its adjoint $ \phi^* $.

The representation $\pi$ maps the algebra to diagonal matrices:
\begin{align}
\label{eq:Z2-A}
f = ( f_+ , f_-) \in \mathbb{C}^2= \mathcal A_{\Z_2} \mapsto 
\begin{bmatrix} 
f_+ & 0   \\
0 & f_- 
\end{bmatrix} 
,
\end{align}
where we have freely identified  $f_{ \pm}$ with the scalar matrices
$ f_{ \pm} I_{ \pm}$, where $ I_{ \pm }$ are the identity matrices 
acting on $ \mathfrak h_{ \pm}$, respectively. 

The differential $[ D_\phi , \cdot ]$ acts as a difference operator:
\begin{align}
\label{eq:D-f}
[ D , f ]  =  
\begin{bmatrix} 
   0 & - \phi ( f_{+} - f_- ) \\  \phi^*  ( f_{+} - f_- ) & 0
\end{bmatrix} 
= ( f_{+} - f_- ) [ D_\phi , e_+] 
\end{align}
where  $ e_+ =\diag (1,0) \in \mathcal A_{\Z_2}$.

In this case, the space of differential forms are all finite dimensional 
and so by counting dimension 
we see that the junk forms $J_D^l (\mathcal A_{\Z_2})$ and junk tensors 
$JT_D^l ( \mathcal A_{\Z_2})$, $l=1,2$, are all zero.
Furthermore, it follows from \eqref{eq:D-f} that
$ \Omega^1_{ D_\phi } ( \mathcal A_{\Z_2})$,  $ T^2_{ D_\phi } ( \mathcal A_{\Z_2})  $
and $ \Lambda^2_{ D_\phi } ( \mathcal A_{\Z_2})$ are all free left $\mathcal A_{\Z_2}$-modules of rank one, generated by 
$ \eta $,  $ \eta \otimes_{ \mathcal A_{\Z_2}} \eta$ and $ \eta^2$ respectively, where 
\begin{align}
\label{eq:beta1}
\eta
=
[D_\phi , e_+ ]
= 
\begin{bmatrix} 
0 & - \phi \\ \phi^* & 0
\end{bmatrix} 
, \, \,
\eta^2
=
[D_\phi , e_+ ]^2
= 
\begin{bmatrix} 
- \phi \phi^* & 0 \\ 0 & - \phi^* \phi 
\end{bmatrix} \in \mathcal A_{\Z_2}.
\end{align}

Due to the Leibniz rule, any connection 
$ \nabla : \Omega^1_{ D_\phi } ( \mathcal A_{\Z_2}) \to T^2_{ D_\phi} ( \mathcal A_{\Z_2})$
is determined by its evaluation on $\eta$:
\begin{align}
\label{eq:nab-Z2}
\nabla \eta = c (\eta \otimes \eta)
,\, \, \text{where $c = \diag(c_+ , c_-) \in \mathcal A_{\Z_2}$. }
\end{align}
Thus the space of connections on $ \Omega_{ D_\phi} ( \mathcal A_{Z_2})$ 
is two-dimensional parametrized by two complex coefficients  $ c_{ \pm }$  as above.

Let us compute the torsions in Definition \ref{defn:tor}.
First, in the differential calculus of \eqref{eq:d2-sig},
the projection map $\sigma_2$ is the identity map as there are no non-zero
junk forms. As $\eta = \delta_D (e_+)$ is an exact form,
$ \mathrm d_{ \sigma_2}\eta = 0$, given a connection $ \nabla $ 
we get for \eqref{eq:TorSwd}:
\begin{align*}
T^\nabla_{ \sigma_2 } (\eta) 
=
\sigma_2 \brac{ (\nabla \eta)_{ (1)} \nabla \eta)_{ (0)} } 
=
c  \eta^2 ,
\end{align*}
where the matrix form of $\eta^2$ is given in \eqref{eq:beta1}.

As the torsion map $ T^\nabla_{ \sigma_2}$ is left 
$ \mathcal A_{\Z_2}$-linear depending on the connection, we have proved 
\begin{lemma}
Let $\nabla$ be a connection defined by \eqref{eq:nab-Z2},
then for $\beta = h \eta \in \Omega^1_{ D_\phi } ( \mathcal A_{Z_2})$,
with $h \in \mathcal A_{Z_2}$,
\begin{align*}
 T^\nabla_{ \sigma_2 } ( \beta) = \beta\eta c .
\end{align*}
\end{lemma}

Note that when we choose $ c= \diag (1 ,-1)$ in \eqref{eq:nab-Z2}, 
then
\begin{align*}
\eta c = D_\phi 
\end{align*}
and hence 
\begin{align}
\label{eq:T-beta}
T^\nabla_{ \sigma_2 } ( \beta)
 = \beta D_\phi .
\end{align}
Moreover, for the differential calculus in \eqref{eq:Psi-cplx}, 
the idempotent $\Psi $ has to be zero,
and we compute the torsion using \eqref{eq:TorSwd-Y}:
\begin{align*}
T^\nabla_{\Psi} (\eta) 
=
(1-\Psi) \brac{ (\nabla \eta)_{ (1)} \otimes \nabla \eta)_{ (0)} } 
=
c  \eta \otimes \eta .
\end{align*}
After applying the multiplication map, we see that  
$\forall w \in \Omega^1_{ D_\phi} ( \mathcal A_{Z_2})$,
\begin{align*}
m \brac{ T^\nabla_{ \Psi} ( w) }  = T^\nabla_{ \sigma_2} (w)
,\, \,
w \in \Omega^1_{ D_\phi} ( \mathcal A_{Z_2}),
\end{align*}
hence we obtain the same trilinear torsion functional
$ \mathscr T_{ \Psi} = \mathscr T_{ \sigma_2}$.
We can formulate the discussion of this subsection as follows.

\begin{thm}\label{z2agree}
Consider the connection $ \nabla $  \eqref{eq:nab-Z2} with
$c = \diag (-1,1)$, the then associated torsion functional
$\mathscr T^{\nabla}_{ \sigma_2} $ defined in \eqref{eq:Tsig}
and $ \mathscr T_{ \Psi}$ defined in \eqref{eq:TPsi} are equal,
and both agree with the spectral torsion functional in Definition \ref{defn:T-DSZ}:
\begin{align*}
\mathscr T^{\nabla}_{ \sigma_2 } ( u , v, w) 
= \mathscr T_{ \Psi} ( u , v, w) 
= \Tr ( u v w D_\phi)
=  \mathscr T_{D_\phi} ( u , v ,w),
\end{align*}
where $u , v ,w \in \Omega^1 ( \mathcal A_{Z_2})$.
\end{thm}
\begin{rmk}\label{rmk:pertz2}
Concerning the connections \eqref{eq:nab-Z2}, the torsion-full one $\nabla$ with $c = \diag (1,-1)$ equals the unique torsion-free (Grassmann) one with $c = \diag (0,0)$ plus a perturbation $S$ corresponding to the torsion and determined by
$S(\eta)= \nabla (\eta)$.
\end{rmk}

\subsection{Tensor Product Construction}
We pass now to the tensor product of two spectral triples 
$ ( \mathcal A_1 , \mathcal H_1 , D_1 , \gamma)$
and $ ( \mathcal A_2 , \mathcal H_2 , D_2)$,
where $ \gamma_1  = \gamma_1^*$ with $ \gamma_1^2 =1$, is a grading
operator of the spectral triple of $ \mathcal A_1$,
which is given by:
\begin{align}
\label{eq:sptrA}
( \mathcal A , \mathcal H , D) = 
\brac{
\mathcal A_1 \otimes \mathcal A_2 ,  
\mathcal H_1 \otimes \mathcal H_2 ,
D_1 \otimes 1 + \gamma_1 \otimes D_2
}.
\end{align}

Let us discuss related construction of the three ingredients:
noncommutative residue (cf. \eqref{eq:nres}),  
connections (cf. \SS\ref{sec:algebraicT}) and 
Hermitian structures (cf. Definition \ref{defn:A-iprd})
that are required for studying the torsion functionals. 

In later computation,
our example concerns a special case in which the 
second spectral has a finite dimensional Hilbert space $ \mathcal H_2$.
In this situation, let $n_1$ be the summability of the first spectral triple,
as  $D_1$ and $ \gamma_1 $ anti-commute, we have 
\begin{align}
\label{eq:D^2}
D^2 = D_1^2 \otimes 1 + 1 \otimes D_2^2  ,
\end{align}
and the spectral triple $ ( \mathcal A , \mathcal H , D)$ above is $n_1$-summable.
More importantly, the Dixmier trace $\Tr^+$ of operators on  $ \mathcal H$
can be factored as follows: for any $  T_j \in B ( \mathcal H_j)$, $j = 1,2$
\begin{align}
\label{eq:Tr+n2=0}
\Tr^+ \brac{ (T_1 \otimes T_2)  \abs D^{-n}  }
=
\Tr^+ \brac{ T_1 \abs{D_1}^{-n} }  \Tr (T_2),
\end{align}
where $\Tr ( \cdot )$ stands for the ordinary trace. 

When $ \mathcal H_2$ is also infinite dimensional, deeper results related
to the Dixmier trace are required. 
We mention some results from \cite[Ch. 6, \SS3]{Con-94},
but the study of the notion of torsion on
these type of examples is out of the scope of the paper. 

Denote by $n_j \in (0,\infty)$ the summability of  $ \mathcal A_j$, $j = 1,2$, 
then \eqref{eq:D^2} holds true as well,
which implies that the spectral of $ \mathcal A$ above is
$ n = n_1 + n_2 $-summable.
For the analogue of \eqref{eq:Tr+n2=0}, we need further assumptions to ensure that
operators appearing there  are measurable.
For example,  $n_1  \ge 1$ and $n_2 \ge 1$ is sufficient, with that, we have
\begin{align*}
c_{ n_1, n_2}
\Tr^+ \brac{ (T_1 \otimes T_2)  \abs D^{-n}  }
=
\Tr^+ \brac{  T_1 \abs D_1^{-n_1}  }
\Tr^+ \brac{  T_2 \abs D_2^{-n_2}  },
\end{align*}
where the constant factor is given by:
\begin{align}
\label{eq:Tr+n1n2}
c_{ n_1, n_2} =
\frac{\Gamma ( n /2 +1)}{\Gamma ( n_1 /2 +1) \Gamma ( n_1 /2 +1)}
.
\end{align}

%

The space of one-forms is decomposed as a sum of $ \mathcal A$-bimodules
\begin{align}
\label{eq:1fm-dcp} 
\Omega_D^1 ( \mathcal A)  = \mathcal E_1 + \mathcal E_2  
, \, \,
\mathcal E_1 = \Omega^1_{  D_1 } ( \mathcal A_1 ) \otimes \mathcal A_2
, \, \,
\mathcal E_2 = \gamma  \mathcal A_1\otimes \Omega^1_{  D_2 } ( \mathcal A_2)
.
\end{align}

Given two  connections 
$ \nabla^{(j)} : \Omega^1_{  D_j } ( \mathcal A_j ) 
\to
\Omega^1_{  D_j } ( \mathcal A_j ) 
\otimes_{ \mathcal A_j } \Omega^1_{  D_j } ( \mathcal A_j ) $, $j = 1, 2$,
with 
\begin{align*}
\nabla^{(j)} w 
= 
\brac{ \nabla^{(j)} w }_{ (1)} \otimes
\brac{ \nabla^{(j)} w }_{ (0)},
\end{align*}
we  define a product-type connection on one-forms of $ \mathcal A$ 
\begin{align}
\label{eq:NA}
\nabla :  \Omega^1_D ( \mathcal A) \to \Omega^1_D ( \mathcal A)
\otimes_{ \mathcal A} \Omega^1_D ( \mathcal A)
\end{align}
by setting for $ w_1 \otimes a_2 \in \mathcal E_1$, that is 
$ w_1 \in \Omega_{ D_1}^1(\mathcal A_1)$,
$ a_2 \in \mathcal A_2 $:
\begin{align}
\label{eq:N-1}
\nabla ( w_1 \otimes a_2)
    :=& 
\sbrac{ \brac{ \nabla^{(1)} w_1 }_{ (1)} \otimes 1 }
\otimes_{ \mathcal A} 
\sbrac{  \brac{ \nabla^{(1)} w_1 }_{ (0)} \otimes a_2}  
\\
+&
( \gamma \otimes [D_2 ,a_2] ) 
\otimes_{ \mathcal A}
(w_1 \otimes 1)
, 
\nonumber
\end{align}
while for $  \gamma a_1 \otimes \omega \in \mathcal E_2$, with 
$ a_1 \in \mathcal A_1$ and $ w_2 \in  \Omega^1_{ D_2}(\mathcal A_2)$
\begin{align}
\label{eq:N-2}
\nabla (  \gamma a_1 \otimes w_2 )
=&\, \,
\brac{ [D_1 , a_1] \otimes 1 } 
\otimes_{ \mathcal A }
(\gamma \otimes w_2 ) 
\\
+&
\sbrac{ \gamma a_1 \otimes \brac{ \nabla^{(2)} w_2 } _{ (1)} }
\otimes_{ \mathcal A }
\sbrac{ \gamma \otimes  \brac{ \nabla^{(2)} w_2 } _{ (0)}}.
\nonumber
\end{align}

Note that, for $j=1,2$, both $ \Omega_{ D_j }( \mathcal A_j)$ and $ \mathcal A_j$
are pre-Hilbert $\mathcal A_j$-modules whose $\mathcal A_j$-valued inner
products can be assembled together to form
a $\mathcal A$-valued inner product on $ \mathcal E_j$.
Explicitly, we have on 
$ \mathcal E_1= \Omega^1_{  D_1 } ( \mathcal A_1 ) \otimes \mathcal A_2 $:
\begin{align}
\label{eq:E1-iprd}
{}_{ \mathcal A}\abrac{ w_1 \otimes P , w_2 \otimes Q}
\defeq
\lrabrac{ \mathcal A_1}{ w_1 , w_2 }{} \otimes P Q^* , 
\, \,
\end{align}
where $ w_1 , w_2 \in \Omega^1_{  D_1 } ( \mathcal A_1 )$ and 
$ P, Q \in \mathcal A_2 $.
Similarly, on 
$ \mathcal E_2 = \gamma  \mathcal A_1\otimes \Omega^1_{  D_2 } ( \mathcal A_2)$:
\begin{align}
\label{eq:E2-iprd}
\lrabrac{\mathcal A}{ \gamma f_1 \otimes u_1 , \gamma f_2 \otimes u_2}{} 
\defeq
f_1 f_2^*  \otimes
\abrac{ u_1 , u_2 }_{ \mathcal A_2},
\end{align}
where $f_1 ,f_2 \in \mathcal A_1$ 
and $ u_1 , u_2 \in \Omega^1_{  D_2 } ( \mathcal A_2)$.

Finally, we obtain the desired pre-Hilbert module structure on 
$ \Omega_D^1 ( \mathcal A)$
by requiring that $ \mathcal E_1 \perp \mathcal E_2$, in other words,
the decomposition \eqref{eq:1fm-dcp} is orthogonal.

\subsection{The almost commutative $M \times \Z_2$ }
Let us apply the construction above to the following two spectral triples. 
The first one is a spin spectral triple
$$(C^\infty(M),H_M,D_M)$$ 
of a closed spin manifold $M$
with the spinor Hilbert space $H_M=L^2(\Sigma)$ and the spinor Dirac operator $D_M$.
The second one is a spectral triple of the two-point space discussed in
\SS\ref{sec:2pt}
\begin{align*}
    \brac{ \mathcal A_{\Z_2}, \mathfrak H ,D_\phi }, 
\end{align*}
defined in \eqref{eq:kcyc0}, where 
$ \mathfrak H = \mathfrak H_+ \oplus \mathfrak H_-$
(the dimensions of $ \mathfrak H_{ \pm}$ can be different), and 
$ \phi : \mathfrak H_+ \to \mathfrak H_-$ with its adjoint 
$ \phi^* : \mathfrak H_- \to \mathfrak H_+ $.

The almost commutative manifold $M \times \Z_2$ refers to the spectral triple
$\brac{ \mathcal A , \mathcal H, D }$
given by the tensor product 
\begin{align}
\label{eq:MZ}
\brac{ \mathcal A , \mathcal H, D }
=
\brac{ \mathcal A_1 \otimes \mathcal A_2 , \mathcal H_1 \otimes \mathcal H_2, D }
,
\end{align}
where $ \mathcal A_1 = C^\infty(M)$, and $ \mathcal A_2 = \mathcal A_{\Z_2}$,
also, $ \mathcal H_1 = H_M$ and $ \mathcal H_2 = \mathfrak H$.
By taking, in \eqref{eq:sptrA},
$ D_1 = D_M$, $D_2 = D_\phi$, and the grading operator
$\gamma_1 := \gamma : H_M \to H_M $ being the one acting on the spinors, 
we can write down the Dirac operator  
$ D = \mathcal D_{ 1} + \mathcal D_{ 2}$ where
\begin{align}
\label{eq:D-MZ}
\mathcal D_{ 1} = D_1 \otimes 1 = D_M \otimes 1
,\, \,
\mathcal D_{ 2} =   
\gamma \otimes D_2 = \gamma \otimes D_\phi  
.
\end{align}

Regarding the decomposition of $ \mathfrak H$,
we have 
$ \mathcal H = (H_M  \otimes \mathfrak H_+) \oplus (H_M \otimes \mathfrak H_-)$, 
and elements of $ \mathcal A$ are  represented as diagonal matrices
\begin{align*}
\diag ( f^+ , f^-)  = f^+ \otimes \pi_{D_2} (e_+ ) + f^- \otimes \pi_{D_2}(1-e_+) ,
\end{align*}
where $f^+ , f^- \in \mathcal A_1$ are smooth functions on the manifold $M$,
and  $e_+ = \diag(1,0) \in \mathcal A_2$.
The matrix form of $D$ is given by
\begin{align*}
D = D_M \otimes 1 + \gamma \otimes D_\phi =
\begin{bmatrix} D_M & 0 \\ 0 & D_M \end{bmatrix} 
+
\begin{bmatrix} 0 &  \gamma \phi \\  \gamma \phi^* & 0 \end{bmatrix} 
=
\begin{bmatrix} D_M & \gamma \phi  \\ \gamma \phi^* & D_M \end{bmatrix} 
.
\end{align*}
Therefore $ \mathcal E_1$ 
consists of diagonal matrices of differentials forms on $M$:
\begin{align*}
\mathcal E_1 = \set{
\begin{bmatrix} w^+ & 0 \\ 0 & w^- \end{bmatrix} 
=
w^+ \otimes e_+ + w^- \otimes (1-e_+)
: w^+ , w^- \in \Omega^1_{ D_1 } ( \mathcal A_1)
    } ,
\end{align*}
while 
\begin{align*}
\mathcal E_2  = \set{
\begin{bmatrix} 0 &  \gamma \phi f^+ \\  \gamma \phi^* f^- & 0 \end{bmatrix} 
: f^+, f^- \in \mathcal A_1
}
\end{align*}

\subsubsection{Spectral Torsion Functional}
Let us see how the spectral torsion functional  fits with the tensor product structure, starting with the noncommutative residue.
As the Hilbert space $ \mathcal H_2$ of the second spectral triple is finite
dimensional, the result of the Dixmier trace in \eqref{eq:Tr+n2=0} can be rephrased
to the following: for bounded operators  $ Q_j \in B(\mathcal H_j)$, $j=1,2$,
\begin{align}
\label{eq:ncres-prd-0}
\nres 
\brac{ (Q_1 \otimes Q_2) \abs D^{-n} } 
=
\nres^1 
\brac{ Q_1  \abs{ D_1 }^{-n} } 
\Tr (Q_2)
,
\end{align}
where $D_1 = D_M$ is the spinor Dirac and $n = \dim M$.
In terms of the volume functional,
\begin{align}
\label{eq:ncres12}
\int^{\nres}  Q_1 \otimes Q_2  = \Tr Q_2 \int^{\nres_1} Q_1 ,
\end{align}
where $\nres^1$ and
$ \int^{\nres_1} \bullet := \nres ( \bullet \abs D_M^{-n})$, 
stands for the noncommutative residue and
the volume functional of the first spectral triple.

It has been shown in \cite[\SS4.2]{DSZ24} that the spectral torsion
functional is non-zero for a particular choice of the three 1-forms.
We complete their calculations for arbitrary three 1-forms
and rephrase them in a way that can be compared with the algebraic torsional functionals. We take also this opportunity to 
extend the results of \cite[\SS4.2]{DSZ24} to 
the case when $\phi$ is not just a complex parameter but a linear map $\phi: \mathbb{C}^k\to \mathbb{C}^\ell$.

\begin{thm}
\label{thm:SpT-MZ}
Let $T_D : \Omega_D (\mathcal A) \to B ( \mathcal H)$ be the left
$ \mathcal A$-module map
\begin{align}
\label{eq:T_D}
T_D (w) := w \mathcal D_{ 2}
,
\end{align}
where $\mathcal D_{ 2} = \gamma \otimes D_\phi$ 
is the second part of the $D$ defined in \eqref{eq:D-MZ},
\begin{enumerate}[$1)$]
\item 
Then the spectral torsion functional of \eqref{eq:MZ}
is given by
\begin{align}
\label{eq:SpT-MZ}
\mathscr T^D ( u , v ,w) 
= 
\int^{\nres} u v T_D (w)
=  
\int^{\nres} u v w \mathcal D_{ 2}
\end{align}
for all $ u , v, w \in \Omega^1_D (\mathcal A)$.
\item 

If  the one-forms are given by elementary tensors
\begin{align}
\label{eq:uvw}
u = u_1 \otimes u_2,
\, \,
v= v_1 \otimes v_2,
\, \,
w = w_1 \otimes w_2
\end{align}
viewed as operators in 
$ B(\mathcal H_1) \otimes B ( \mathcal H_2) 
\subset B ( \mathcal H_1 \otimes \mathcal H_2)$,
then by taking \eqref{eq:ncres12} into account, 
\eqref{eq:SpT-MZ} is equal to
\begin{align}
\label{eq:SpT-MZ-prd}
\mathscr T^D ( u , v ,w) 
= 
\Tr \brac{ u_2 v_2 w_2 D_\phi } 
\int^{\nres_1}
u_1 v_1 w_1 
\gamma 
.
\end{align}
\end{enumerate}
\end{thm}
\begin{proof}
By the definition \eqref{eq:T-DSZ}, we are looking at functionals of the form 
\begin{align}
\label{eq:Qfunl}   
 Q \to \nres \brac{ Q D \abs D^{-n} }, \, \, 
\end{align}
in which  $ Q D$ has order at most one
regarding the underlying pseudodifferential calculus. 
Recall from \eqref{eq:D^2} that
\begin{align*}
 D^2 = D_M^2 \otimes 1 + 1 \otimes D_\phi^2
 = D_M^2 \otimes 1 \brac{ 1 \otimes 1 + D_M^{-2} \otimes D_2^2 }
\end{align*}
due to the fact that the grading $ \gamma $ and $ D_M $ anti-commute.
It leads to the expansion: 
\begin{align}
\label{eq:D-expn}
\abs  D^{-n} = (D^{-2})^{n/2}
= 
( D_M^2 \otimes 1)^{-n/2}
\brac{ 1 \otimes 1 - \frac{n}{2} D_M^{-2} \otimes D_2^2  + \ldots}
.
\end{align}
Only the first term above 
gives nontrivial contribution to the functional in \eqref{eq:Qfunl} so that
\begin{align*}
&\, \,
\nres \brac{ Q D \abs D^{-n} } 
=
\nres \brac{ Q D \abs{D_M \otimes 1}^{-n} } 
\\
=
&\, \,
\nres \brac{ Q (D_M \otimes 1) \abs{D_M \otimes 1}^{-n} } 
+
\nres \brac{ Q (\gamma \otimes D_2)\abs{D_M \otimes 1}^{-n} } 
.
\end{align*}

The functional \eqref{eq:SpT-MZ} corresponds to the case in which  
$ Q = u v w = u_1 v_1 w_1 \otimes u_2 v_2 w_2$ 
is the product of the three one-forms.  

With  \eqref{eq:ncres-prd-0} in mind, we claim that  the first term above 
\begin{align*}
\nres \brac{ Q (D_M \otimes 1) \abs{D_M \otimes 1}^{-n} } 
&=
\nres \brac{ u_1 v_1 w_1 D_M  \abs{D_M }^{-n}} 
\Tr \brac{ u_2 v_2 w_2 } =0.
\end{align*}
In fact, it requires a stronger property, called spectral closed 
cf. \cite[Lemma 3.3]{DSZ24} \footnote{the notion was first 
introduced in  \cite{DSZ23}},
of the spinor spectral triple, namely, 
\begin{align*}
\nres \brac{  Q D_M \abs{ D_M}^{-n}   } = 0
\end{align*}
for any zero-order pseudodifferential operator $Q$. 

Finally, the second term yields right hand side of
\eqref{eq:SpT-MZ} and \eqref{eq:SpT-MZ-prd}:
\begin{align*}
&
\nres \brac{ Q (\gamma \otimes D_\phi )\abs{D_M \otimes 1}^{-n} } 
=
\Tr \brac{ u_2 v_2 w_2 D_\phi } 
\nres \brac{  u_1 v_1 w_1 \gamma \abs{D_M}^{-n}}
\\
&=
\Tr \brac{ u_2 v_2 w_2 D_\phi } 
\int^{\nres_1}
u_1 v_1 w_1 \gamma 
=
\int^{\nres}
u v w \mathcal D_{ 2}
,
\end{align*}
where we used \eqref{eq:ncres12} for the last step. 
\end{proof}

\subsection{The Projection $\Psi$}
The goal of the section is to construct a differential calculus in the approach
of Mesland-Rennie described in \SS\ref{sec:MR-cal}.
The essential ingredient (cf. \eqref{eq:Psi-cplx}), is a projection 
$\Psi: T^2_D (\mathcal A) \to T^2_D (\mathcal A)$ such that 
\begin{align*}
JT^2_D \subseteq \textrm{Im} (\Psi) \subseteq m^{-1}(J^2_D).
\end{align*}

We also remind that the $ \mathcal A_1$-valued inner product on
one-forms of $M$ is obtained by complexifying the
underlying Riemannian metric $g$ to a Hermitian one $g_{ \mathbb{C}}$:
\begin{align*}
\abrac{ \omega_1 , \omega_2}_{ \mathcal A_1 }
\defeq
\abrac{ \omega_1 , \omega_2 }_{ g_{ \mathbb{C}}} 
.
\end{align*}
For $ \Omega_{ D_\phi}^1  ( \mathcal A_2)$, the elements,
say $ u_1 , u_2$, are represented as off-diagonal matrices acting on $\mathcal H_2$, 
thus the $\mathcal A_2$-valued inner product is simply given by the following
matrix multiplication:
\begin{align*}
\abrac{ u_1 , u_2}_{ \mathcal A_2} \defeq u_1 u_2^*.
\end{align*}

Although $ \mathcal A$ is a commutative algebra, the quantum nature of 
the almost commutative manifold is derived from the fact that
$ \mathcal E_2$ is not a symmetric $\mathcal A$-bimodule 
\footnote{ $ \mathcal E_1$ is indeed a symmetric $\mathcal A$-bimodule}.
It leads us to consider the flipping map on  functions on the two-point space:
\begin{align}
\label{eq:flip-A2}
\alpha : \mathcal A_2 \to \mathcal A_2:
\diag (f ,g )   \mapsto \diag (g,f)
.
\end{align}
It is an algebra homomorphism, in particular a $\mathcal A_2$-bimodule map, 
and $\alpha^2 =1$.
The induced map
\begin{align}
\label{eq:til-alp}
\tilde \alpha \defeq 1 \otimes \alpha :
\mathcal A_1 \otimes \mathcal A_2 \to
\mathcal A_1 \otimes \mathcal A_2
\end{align}
interchanges   the right and the left actions of
$ \mathcal A$ on $ \mathcal E_2$:
\begin{align*}
u \cdot \tilde f
=
(1 \otimes \alpha) ( \tilde f) \cdot u 
,\, \,
u \in \mathcal E_2
, \, \,
\tilde f \in \mathcal A.
\end{align*}
Moreover, regarding the inner product \eqref{eq:E1-iprd}, we have:
\begin{lemma}
Denote by 
$ \alpha_1 \defeq (1 \otimes \alpha) : \mathcal E_1 \to \mathcal E_1$,
where we recall that
$ \mathcal E_1 = \Omega^1_{ D_1} ( \mathcal A_1) \otimes \mathcal A_2$. Then, for all $ u \in \mathcal E_2$,
\begin{align}
\label{eq:r-l-innprd}
u \cdot \lrabrac{\mathcal A}{  \alpha_1 (x) ,y }{} 
=
\lrabrac{\mathcal A}{ x,  \alpha_1(y) }{} \cdot u.
\end{align}
\end{lemma}
\begin{proof}
Write $x = \omega \otimes P$ and $y = \mu \otimes Q$ where 
$ \omega, \mu \in \Omega^1_{ D_1} ( \mathcal A_1)$
and 
$ P, Q \in \mathcal A_2$.
By definition 
\begin{align*}
\lrabrac{\mathcal A}{  \alpha_1(x) ,y }{}    
=
\lrabrac{\mathcal A_1}{ \omega, \mu}{} \otimes \alpha(P) Q^*    
\end{align*}
thus
\begin{align*}
u \cdot 
\lrabrac{\mathcal A}{  \alpha_1(x) ,y }{}    
=
u \cdot
\lrabrac{\mathcal A_1}{ \omega, \mu}{} \otimes \alpha(P) Q^*    
=
\lrabrac{\mathcal A_1}{ \omega, \mu}{} \otimes \alpha \brac{ \alpha(P) Q^*  }
\cdot u.
\end{align*}
To conclude the proof, we observe that 
$\alpha \brac{ \alpha(P) Q^*  } = P \alpha \brac{ Q^*} $ and 
\begin{align*}
\lrabrac{\mathcal A_1}{ \omega, \mu}{} \otimes 
P \alpha(Q^*  ) 
=
\abrac{ x , (1 \otimes \alpha) (y)}
\end{align*}
\end{proof}

We will need the following $ \mathcal A$-bimodule maps derived from the flip
$\alpha$:
\begin{align*}
        &
\beta_{(11)}:
\mathcal E_1 \otimes_{ \mathcal A}  \mathcal E_1 
\to
\mathcal E_1 \otimes_{ \mathcal A}  \mathcal E_1:
x \otimes y \mapsto 
y \otimes x,
\\
    &
\beta_{(12)}:  \mathcal E_1 \otimes_{ \mathcal A}  \mathcal E_2
\to
\mathcal E_2 \otimes_{ \mathcal A}  \mathcal E_1: 
x \otimes u 
\mapsto
u \otimes  \alpha_1 (x)
,
\\
    &
\beta_{(21)}:
\mathcal E_2 \otimes_{ \mathcal A}  \mathcal E_1 
\to
\mathcal E_1 \otimes_{ \mathcal A}  \mathcal E_2:
u \otimes x 
\mapsto
\alpha_1 (x) \otimes u
.
\end{align*}
Check they are well-defined maps over the balanced tensor  $ \otimes_{ \mathcal A}$.

The pre-Hilbert $\mathcal A$-module structure on 
$ T_D^2 ( \mathcal A) =
\Omega^1_D ( \mathcal A) \otimes_{ \mathcal A} \Omega^1_D ( \mathcal A)$
can be defined via the standard construction in the theory of correspondence:
\begin{align}
\label{eq:inprd-OmgA}
\lrabrac{ \mathcal A}{ x \otimes y , u \otimes v }{} 
\defeq
\lrabrac{ \mathcal A}{ x  \cdot \lrabrac{\mathcal A}{ y, v}{} ,  u}{}
.
\end{align}

\begin{lemma}
\label{lem:betamaps}
Regarding the $\mathcal A$-valued inner product, we have
\begin{enumerate}
\item
$ \beta_{ 11}$ is self-adjoint and $ \beta_{ 11}^2 = 1$$;$   
\item 
$\beta_{ (21)} = \beta_{ (12)}^* $
and
\begin{align}
\label{eq:beta^2}
\beta_{ (21)} \beta_{ (12)} = 1_{ \mathcal E_1 \otimes_{ \mathcal A} \mathcal E_2 } 
,\, \,
\beta_{ (12)} \beta_{ (21)} = 1_{ \mathcal E_2 \otimes_{ \mathcal A} \mathcal E_1 } 
.
\end{align}
\end{enumerate}

\end{lemma}
\begin{proof}
The property \eqref{eq:beta^2} is obvious. 
Let us check that $\beta_{ (12)}$ and $\beta_{ (21)}$ are indeed adjoint to
each other.
For $ x ,y \in \mathcal E_1$ and $ u, v \in \mathcal E_2$, we compute:
\begin{align*}
&
\lrabrac{\mathcal A}{ 
\beta (x \otimes_{ \mathcal A} u), v \otimes_{ \mathcal A}  y
}{} 
=
\lrabrac{\mathcal A}{ 
u \otimes_{ \mathcal A} \alpha_1(x) ,
v \otimes_{ \mathcal A}  y
}{}
\\
=&
\lrabrac{\mathcal A}{ 
u 
\lrabrac{\mathcal A}{ 
 \alpha_1(x) , y   
}{}
,
v
}{}
=
\lrabrac{\mathcal A}{ 
 x , \alpha_1(y)
}{}
\lrabrac{\mathcal A}{ 
    u ,v
}{},
\end{align*}
where we have used  \eqref{eq:r-l-innprd} in last step.
Similarly, for the other side
\begin{align*}
    &
\lrabrac{\mathcal A}{ 
x \otimes_{ \mathcal A} u, \beta^* \brac{  v \otimes_{ \mathcal A}  y} 
}{}
=
\lrabrac{\mathcal A}{ 
x \otimes_{ \mathcal A} u, 
 \alpha_1 (y) \otimes_{ \mathcal A} v
}{}
\\
=&
\lrabrac{\mathcal A}{ 
x 
\cdot
\lrabrac{\mathcal A}{ u,v }{}
, \alpha_1(y)
}{}
=
\lrabrac{\mathcal A}{ u,v }{}
\lrabrac{\mathcal A}{ x , \alpha_1(y) }{}
,
\end{align*}
where we need the symmetric bimodule property of $ \mathcal E_1$ in the last
step.
The agreements follows from the commutativity of $ \mathcal A$.

The self-adjointness of $ \beta_{ (11)}$ can be verified in a similar manner, 
for any $x, x' , y, y' \in \mathcal E_1$
\begin{align*}
\lrabrac{\mathcal A}{ \beta_{ (11)} \brac{ x \otimes y },  x' \otimes y'}{} 
=
\lrabrac{\mathcal A}{ x , y' }{} 
\lrabrac{\mathcal A}{ y,x' }{} 
\end{align*}
and
\begin{align*}
\lrabrac{\mathcal A}{  x \otimes y , \beta_{ (11)} \brac{ x' \otimes y' } }{} 
=
\lrabrac{\mathcal A}{ y,x' }{} 
\lrabrac{\mathcal A}{ x , y' }{} 
.
\end{align*}
\end{proof}

\begin{prop}
\label{prop:dmp-T_D^2}
With respect to  the decomposition $\eqref{eq:1fm-dcp}$
for $ \Omega^1_D ( \mathcal A)$,
$ T_D^2 ( \mathcal A) $ admits the following orthogonal decomposition regarding 
the $\mathcal A$-valued inner product in \eqref{eq:inprd-OmgA}$:$
\begin{align}
\label{eq:dmp-T_D^2}
T_D^2 ( \mathcal A) 
=
\bigoplus_{ ( i,j), i,j=1,2} \mathcal E_{ (i,j)}
,\, \,
\mathcal E_{ (i,j)},
=
\mathcal E_i \otimes_{ \mathcal A} \mathcal E_j.
\end{align}
\end{prop}
\begin{proof}
The results follows from the claim that in the definition \eqref{eq:inprd-OmgA},
if one of the pairs 
$ \lrabrac{\mathcal A}{ x ,u }{}$ and
$ \lrabrac{\mathcal A}{ y ,v }{}$ is zero, then the resulting inner product  
on the right hand side is  zero.
\end{proof}

Now we are ready to define the projection
$\Psi : T^2_D (\mathcal A) \to T^2_D (\mathcal A)$
with respect to the decomposition:
\begin{align*}
 T^2_D (\mathcal A) = \mathcal F_1 \oplus \mathcal F_2, 
 \, \,
 \mathcal F_1 = \mathcal E_{( 1,1)} \oplus \mathcal E_{ (2,2)},
 \, \,
 \mathcal F_2 = \mathcal E_{( 1,2)} \oplus \mathcal E_{ (2,1)},
\end{align*}
on $ \mathcal F_1$:
\begin{align}
\label{eq:Y-F1}
\Psi =
\frac{1}{2} \begin{bmatrix} 
 \brac{ 1 + \beta_{ (11)} } & 0 \\ 0 & 0
\end{bmatrix} 
: \mathcal F_1 \to \mathcal F_1
,
\end{align}
and $ \mathcal F_2$:
\begin{align}
\label{eq:Y-F2}
\Psi = 
\frac{1}{2}
\begin{bmatrix} 
1 & \beta_{ (21)} \\
\beta_{ (12)}& 1 
\end{bmatrix} 
: \mathcal F_2 \to \mathcal F_2
.
\end{align}

The properties $ \Psi = \Psi^*$ and  $ \Psi^2 = \Psi$ are inherited directly
from the corresponding properties of the $\beta$'s in Lemma   \ref{lem:betamaps}.

\subsection{Junk Tensors $JT^2_D (\mathcal A)$}
The main result of the section is to show that the inclusion holds
\begin{align*}
JT^2_D (\mathcal A) \subseteq \mathrm{Im} (\Psi).
\end{align*}
We take advantage of the orthogonal decomposition in \eqref{eq:dmp-T_D^2} and
break the verification  into three parts: Proposition \ref{prop:del22=0},
Corollary \ref{cor:del11}, and Proposition \ref{prop:del21-21}. 

Let $m: \mathcal A \to \mathcal A$ be the multiplication map and 
$\Omega^1_u(\mathcal A) = \ker m$ is the space of universal one-forms.
By definition, $ JT^2_D (\mathcal A) = \delta_D  \brac{ \ker \pi_D } $ 
is the image of 
$ \delta_D : \Omega^1_u( \mathcal A) \to T^2_{ D} ( \mathcal A)$
sending $ w = \sum_{  } f \otimes g \in \Omega^1_u ( \mathcal A) $
to
\begin{align}
\label{eq:delD-dfn}
\delta_D ( w) = 
\sum_{  } [D ,f ] \otimes_{ \mathcal A} [D ,g ],
\in T^2_D ( \mathcal A), 
\end{align}
and $\pi_D$ denotes the representation 
$ \pi_D : \Omega^1_u( \mathcal A) \to \Omega^1_{ D} ( \mathcal A) $
associated with the commutator $[ D, \cdot ]$,
sending  $ \sum_{  } f \otimes g $ to $ \sum_{  } f [D, g] $.

With $D =   \mathcal D_1 + \mathcal D_2 $,
where $\mathcal D_1= D_M \otimes 1$ and
$\mathcal D_2 = \gamma \otimes D_\phi $, 
we decompose 
$\delta_D = \sum_{ i,j \in\set{1,2} }  \delta^{(i,j)}_D$,
where 
$ \delta^{(i,j)}_D : \Omega^1_u( \mathcal A) \to \mathcal E_{ (i,j)}$,
with the notations in \eqref{eq:delD-dfn},
\begin{align*}
 \delta_D^{(i,j)} ( w) = 
\sum_{  } [ \mathcal D_i ,f ] \otimes_{ \mathcal A} [\mathcal D_j ,g ]
\in \mathcal E_{ (i,j)}. 
\end{align*}

As the decomposition \eqref{eq:1fm-dcp} is orthogonal, we have 
$\ker \pi_D = \ker \pi_{ \mathcal D_1} \cap \ker \pi_{ \mathcal D_2}$.
For $a = \diag(a^+, a^-) \in \mathcal A$,
$[\mathcal D_1 , a] = \diag( \mathrm d a^+ , \mathrm d a^- )$ gives rise to 
a diagonal matrix with
differential one-forms of the pair of functions $a^{\pm}$, while
$[ \mathcal D_2 , a]$ yields an off-diagonal matrix with the difference of $ a^{\pm}$
implemented by the operator $ \tilde \alpha - 1$, see \eqref{eq:til-alp}:
\begin{lemma}
Denote by
\begin{align*}
\eta_{ \chi}=
\begin{bmatrix} 0 &  \chi \\ \bar \chi & 0 \end{bmatrix}  
\in \mathcal E_2,
\end{align*}
for any $ a \in \mathcal A$, we have, 
\begin{align}
\label{eq:D2-df}
[\mathcal D_2 , a ] = ( \tilde \alpha - 1 ) (a) \cdot \eta_{ \chi }
\end{align}
in particular:
\begin{align*}
[\mathcal D_2 , \tilde \alpha (a) ] = - [\mathcal D_2 , a  ] .
\end{align*}
\end{lemma}
\begin{proof}
Let $ a= \diag ( a^+ , a^- )$:
\begin{align*}
    &
\sbrac{ \mathcal D_2 ,  
\begin{bmatrix} a^+  & 0 \\ 0 & a^- \end{bmatrix}  }
=
\begin{bmatrix} 
0 &  - \chi  ( a^+ - a^- ) \\ \bar \chi ( a^+ - a^- ) & 0 
\end{bmatrix}  
\\
= &
\begin{bmatrix}
( a^+ - a^- ) & 0 \\ 0 & ( a^+ - a^- )
\end{bmatrix}  
\begin{bmatrix} 0 & - \chi \\ \bar \chi & 0 \end{bmatrix}  
=
\brac{ \tilde \alpha -1  } ( a ) \cdot \eta_{ \chi }.
\end{align*}

For the second equation:
\begin{align*}
[\mathcal D_2 , \tilde \alpha (a) ] = 
( \tilde \alpha - 1 ) \brac{ \tilde \alpha (a) } \cdot \eta_{ \chi }
= ( 1 - \tilde \alpha) (a) \cdot \eta_{ \chi }
= - [\mathcal D_2 , a  ] .
\end{align*}
\end{proof}

\begin{lemma}
For $ w = \sum_{  } f \otimes g  \in \ker \pi_{ \mathcal D_2}$,
we have 
\begin{align}
\label{eq:fD2g}
\sum_{  } f \cdot ( \tilde \alpha - 1) (g) = 0
\\
\sum_{  } ( \tilde \alpha  - 1) (f) \cdot ( \tilde \alpha - 1) (g) = 0
\label{eq:D2-2}
\end{align}
\end{lemma}

\begin{proof}
We will repeatedly using  \eqref{eq:D2-df} to handle the commutator
$[\mathcal D_2, \cdot ]$. 
As $ w \in \ker \pi_{ \mathcal D_2}$, we have
$ 0 = \sum_{  } f [\mathcal D_2 ,g] 
= \sum_{  }  f \cdot ( \tilde \alpha - 1) (g) \eta_{ \chi}$,  
which proves \eqref{eq:fD2g}.

To argue \eqref{eq:D2-2}, we take advantage of the fact that elements of
$ \mathcal E_{ (2,2)}$ are represented as operators (two by two matrices)
\begin{align*}
&
\sum_{  }  [ \mathcal D_2 , f] [\mathcal D_2 ,g] =
\sum_{  }
( \tilde \alpha - 1) (f) 
\cdot \eta_{ \chi} \cdot 
( \tilde \alpha - 1) (g) \eta_{ \chi}
\\
=&
\sum_{  }  
( \tilde \alpha - 1) (f) 
\tilde \alpha \brac{  ( \tilde \alpha - 1) (g)}
\eta_{ \chi}^2
=
- \sum_{  }  
 ( \tilde \alpha - 1) (f) \cdot
( \tilde \alpha - 1) (g) 
\eta_{ \chi}^2
,
\end{align*}
where we have used $ \tilde \alpha^2 =1$. 
It remains to see $ \sum_{  }  [ \mathcal D_2 , f] [\mathcal D_2 ,g]= 0$. 
Indeed, as matrices, we compute $ [\mathcal D_2 , \eta_{ \chi} ] =0$, hence
the iterated commutator reads:
\begin{align*}
[ \mathcal D_2 , [\mathcal D_2, g] ]
&=
[ \mathcal D_2 ,  ( \tilde \alpha - 1) (g) \eta_{ \chi} ] =
[ \mathcal D_2 ,  ( \tilde \alpha - 1) (g)  ]  \eta_{ \chi} =
( \tilde \alpha - 1)^2 (g)  \eta_{ \chi}^2
\\
&=
-2 ( \tilde \alpha - 1) (g)  \eta_{ \chi}^2
.
 \end{align*}
The desired result follows from applying the derivation
$[\mathcal D_2 , \cdot ]$ onto $ \sum_{  } f [\mathcal D_2 ,g ] = 0$,
also with the help of \eqref{eq:fD2g}:
\begin{align*}
\sum_{  }  
[ \mathcal D_2 , f] [\mathcal D_2 ,g] = 
\sum_{  }  
- f [ \mathcal D_2 , [\mathcal D_2, g] ]
=
\sum_{  }  2 f \cdot ( \tilde \alpha - 1) (g)  \eta_{ \chi} =0
.
\end{align*}

\end{proof}

\begin{prop}
\label{prop:del22=0}
For any $w = \sum_{  } f \otimes g \in \ker \pi_{ \mathcal D_2} $, 
we have $ \delta_{ D}^{(2,2)} ( w) = 0$. 
In other words, 
the projection of $ JT^2 (\mathcal A) $ onto $ \mathcal E_{ (2,2)}$
is indeed zero. 
\end{prop}
\begin{proof}

    Using $ [\mathcal D_2 , \eta_{ \chi} ] =0$, we compute
\begin{align*}
[ \mathcal D_2 , [\mathcal D_2, g] ]
&=
[ \mathcal D_2 ,  ( \tilde \alpha - 1) (g) \eta_{ \chi} ] =
[ \mathcal D_2 ,  ( \tilde \alpha - 1) (g)  ]  \eta_{ \chi} =
( \tilde \alpha - 1)^2 (g)  \eta_{ \chi}^2
\\
&=
2 ( \tilde \alpha - 1) (g)  \eta_{ \chi}^2
.
 \end{align*}
 By applying the derivation $[\mathcal D_2 , \cdot ]$ onto $ \sum_{  } f [\mathcal D_2 ,g ] $,
we have
\begin{align*}
\sum_{  }  
[ \mathcal D_2 , f] [\mathcal D_2 ,g] = 
\sum_{  }  
- f [ \mathcal D_2 , [\mathcal D_2, g] ]
=
\sum_{  }  -2 f \cdot ( \tilde \alpha - 1) (g)  \eta_{ \chi} =0
\end{align*}
according to \eqref{eq:fD2g}. On the other hand, we have obtained,
using $ \tilde \alpha ( \tilde \alpha - 1 ) = 1 - \tilde \alpha $:
\begin{align*}
    0 &=  
\sum_{  }  
[ \mathcal D_2 , f] [\mathcal D_2 ,g] =
( \tilde \alpha - 1) (f) 
\cdot \eta_{ \chi} \cdot 
( \tilde \alpha - 1) (g) \eta_{ \chi}
\\
      &=
\sum_{  }  
( \tilde \alpha - 1) (f) 
\tilde \alpha \brac{  ( \tilde \alpha - 1) (g)}
\eta_{ \chi}^2
=
- \sum_{  }  
 ( \tilde \alpha - 1) (f) \cdot
( \tilde \alpha - 1) (g) 
\eta_{ \chi}^2
.
\end{align*}
Finally
\begin{align*}
&
\sum_{ j } [\mathcal D_2 , f_j ] \otimes_{ \mathcal A} [\mathcal D_2 , g_j] 
=
\sum_{ j } 
( \tilde \alpha -1 ) (f_j) \cdot \eta_{ \chi }
\otimes_{ \mathcal A} 
( \tilde \alpha -1 ) (g_j) \cdot \eta_{ \chi }
\\
    =&
\sum_{ j } 
( \tilde \alpha -1 ) (f_j) 
( \tilde \alpha -1 )  \brac{ \tilde \alpha (g_j) }
 \cdot
\eta_{ \chi } \otimes_{ \mathcal A} \eta_{ \chi }
\\
    =&
-\sum_{ j } 
( \tilde \alpha -1 ) (f_j) 
( \tilde \alpha -1 )  (g_j)
\cdot
\eta_{ \chi } \otimes_{ \mathcal A} \eta_{ \chi } =0
.
\end{align*}    
\end{proof}

\begin{lemma}
Let $ w = \sum_{ \mu } f_\nu \otimes g_\mu  \in \ker \pi_{ D_M }  $,
where $ \pi_{ D_M }:  \Omega^1_u ( \mathcal A_1) 
\to \Omega^1_{ D_M} (\mathcal A_1) \otimes_{ \mathcal A_1}
\Omega^1_{ D_M} (\mathcal A_1) 
$.
Then 
\begin{align*}
\sum_{ \mu }
[D_M , f_{ \mu }] \otimes_{ \mathcal A_1} [D_M , g_{ \mu} ]   
=
\sum_{ \mu }
[D_M , g_{ \mu }] \otimes_{ \mathcal A_1} [D_M , f_{ \mu} ]   
\end{align*}
\end{lemma}
\begin{proof}
We identify the sum above as $2$-covectors which are given   in local charts:
\begin{align*}
\sum_{ i,j,\mu }
\partial_{ x_i } f_\mu \partial_{ x_j } g_\mu 
d x_i \otimes d x_j.
\end{align*}
We need to show that it is a symmetric tensor, namely, for fixed $i ,j $,
\begin{align*}
\sum_{ \mu } 
\partial_{ x_i } f_\mu \partial_{ x_j } g_\mu 
= 
\sum_{ \mu } 
\partial_{ x_j } f_\mu \partial_{ x_i } g_\mu .
\end{align*}
Indeed, we have $ \sum_{ \mu } f_\mu d g_\mu =0 $ as $ w \in \ker \pi_{ D_M}$,
thus $ \sum_{ \mu } f_\mu \partial_{ x_j } g_\mu =0 $,
after applying $ \partial_{ x_i }$ on both sides:
\begin{align*}
\sum_{ \mu } \partial_{ x_i }f_\mu \partial_{ x_j } g_\mu
=
-  \sum_{ \mu } f_\mu \partial_{ x_i }\partial_{ x_j } g_\mu =0  .
\end{align*}
The right hand side above is symmetric in $ i,j $, we have completed the proof.
\end{proof}

Same argument as above works without much modification when $D_M$ is 
replaced by  $\mathcal D_1 = D_M \otimes 1$, which proves that the 
$ \mathcal E_{ (1,1)}$ component of $JT^2_D ( \mathcal A)$ is also contained
in the range of $\Psi$:
\begin{corr}
\label{cor:del11}
For $ \omega \in \ker \pi_{ \mathcal D_1}$, where $\mathcal D_1 = D_M \otimes 1$ and 
$ \pi_{ \mathcal D_1} : \Omega^1_u ( \mathcal A) \to \mathcal E_{ (1)}$, 
we have 
\begin{align*}
\delta_{ D}^{(1,1)} (\omega) = \beta_{ (1,1)} \brac{  \delta^{(1,1)}_{ D} (\omega)} ,
\end{align*}
that is $ \delta_{ \mathcal D_1} (\omega) \in \mathrm{Im} (\Psi)$.
\end{corr}

Lastly, let us verify that the 
$ \mathcal F_2 = \mathcal E_{ (1,2)} \oplus \mathcal E_{ (2,1)}$ 
component of $JT^2_D ( \mathcal A)$ is contained in $ \mathrm{Im} (\Psi)$. 
Thanks to Lemma \ref{lem:betamaps}, it is sufficient to prove the following.
\begin{prop}
\label{prop:del21-21}
For $ w \in \ker \pi_{ D} = \ker \pi_{ \mathcal D_1} \cap \ker \pi_{ \mathcal D_2}$, 
we have
\begin{align}
\label{eq:del12-21}
\beta_{ (12)} \brac{ \delta_D^{(1,2)} (w) }  = \delta_{ D}^{(2,1)}
,\, \,
\beta_{ (21)} \brac{ \delta_D^{(2,1)} (w) }  = \delta_{ D}^{(1,2)}
.
\end{align}
\end{prop}
\begin{proof}
We write  $ w = \sum_{  }  f \otimes g$ with $f , g \in \mathcal A$ and
$\sum_{  } f g =0 $. 
As $ w \in \ker \pi_{ \mathcal D_1}$, we see that $ \sum_{  } f [\mathcal D_1 , g]  =0$, 
hence 
$$  \sum_{  } g [\mathcal D_1 , f ] =  \sum_{  } [\mathcal D_1 , f ] g 
= - \sum_{  } f [\mathcal D_1 , g]  =0.$$ 
Let us look at the first equation in \eqref{eq:del12-21},
the left side can be computed as follows:
\begin{align*}
\delta_{ D}^{(2,1)} \brac{ w } 
&=
\sum_{  }
[\mathcal D_2 ,f ] \otimes_{ \mathcal A} [\mathcal D_1 , g] 
=
\sum_{  }
( \tilde \alpha -1)(f) \eta_{ \chi} 
\otimes_{ \mathcal A} [\mathcal D_1 , g]
\\
&=
\sum_{  }
\eta_{ \chi} 
\otimes_{ \mathcal A}
( 1- \tilde \alpha )(f)[\mathcal D_1 , g]
=
\sum_{  }
\eta_{ \chi} 
\otimes_{ \mathcal A}
( - \tilde \alpha )(f)[\mathcal D_1 , g]
\\
&=
\sum_{  }
\eta_{ \chi} 
\otimes_{ \mathcal A}
[\mathcal D_1 , \tilde \alpha (f)] g
,
\end{align*}
for the last step, we need $ w \in \ker \pi_{ \mathcal D_2} $ thus \eqref{eq:fD2g} holds,
and then $ \sum_{  } \tilde \alpha (f) \cdot g = \sum_{  } fg =0  $,
which further yields
$ \sum_{  } \tilde \alpha (f) \cdot [\mathcal D_1 , g] 
= - \sum{} [ \mathcal D_1  , \tilde \alpha (f)] \cdot g $.
While the right hand side reads:
\begin{align*}
    &
\beta_{ (12)} \brac{ 
\sum_{  }
[\mathcal D_1 ,f ] \otimes_{ \mathcal A} [\mathcal D_2 , g] 
} =
\sum_{  }
\beta_{ (12)} \brac{ 
[\mathcal D_1 ,f ] \otimes_{ \mathcal A} ( \tilde \alpha -1)(g) \eta_{ \chi} } 
\\
=&
\sum_{  }
 \eta_{ \chi}  
\otimes_{ \mathcal A}
(1 - \tilde \alpha)(g)
\alpha_1 \brac{ [\mathcal D_1 ,f] }
,
\end{align*}
and the second factor indeed agrees with that of 
$ \delta_{ D}^{(2,1)} \brac{ w }$ above:
\begin{align*}
&
\sum_{  }
(1 - \tilde \alpha)(g) \alpha_1 \brac{ [\mathcal D_1 ,f] }
=
\sum_{  }
\alpha_1 \brac{ ( \tilde \alpha - 1) (g) \cdot [\mathcal D_1 , f] }
\\
=&
\sum_{  }
\alpha_1 \brac{  \tilde \alpha(g) \cdot [\mathcal D_1 , f] }
=
\sum_{  }
g [\mathcal D_1 , \tilde \alpha (f) ]
=
\sum_{  }
 [\mathcal D_1 , \tilde \alpha (f) ] g.
\end{align*}
The second equation in \eqref{eq:del12-21} can be proved in a similar way,
the details are left to the reader.
\end{proof}

\section{Main Results}
Throughout this section, let $ \nabla $ be the product-type connection, 
formally written as 
\begin{align}
\label{eq:nab-prdty}
\nabla = \nabla^{ (1)} \otimes 1 + \gamma \otimes \nabla^{(2)},  
\end{align}
where $ \nabla^{(1)}$ is the Levi-Civita
connection of the spin manifold $M$ and $ \nabla^{(2)}$ is the connection in
Theorem \ref{z2agree} whose torsion agrees with the spectral one.
The precise meaning of the right hand side is given in
\eqref{eq:N-1} and \eqref{eq:N-2}. 

We first compute the algebraic torsion $T^\nabla_{\sigma_2}$ and
$ T_\Psi^{\nabla}$ regarding the two differential calculi \eqref{eq:C-cplx} and
\eqref{eq:Psi-cplx}, and then try to recover the spectral functional computed
in Theorem \ref{thm:SpT-MZ}.

\subsection{Algebraic Torsion in Connes' Calculus}
\label{sec:AT-C}
Let us give short computation of the algebraic torsion 
regarding the differential calculus in \eqref{eq:C-cplx}. 
We first need to work out the
map $\sigma_2 \circ m$ in the definition of $T^\nabla_\sigma$ in  \eqref{eq:TorL}.
Roughly speaking, on the manifold part, this map is given by taking the 
leading term of the Clifford multiplication $m$ on one-forms $ \Lambda^1(M)$:
\begin{align*}
m: \Lambda^1 (M) \otimes \Lambda^1 (M) \to B( \mathcal H_1),\, \,  
w_1 \otimes w_2 \to \mathbf c ( w_1 ) \mathbf c ( w_2 ), 
\end{align*}
so that 
\begin{align*}
\sigma_2 \circ m:
\Lambda^1 (M) \otimes \Lambda^1 (M) \to B( \mathcal H_1),\, \,  
w_1 \otimes w_2 \to \mathbf c ( w_1 \wedge w_2)
\end{align*}
where $ \wedge$ is the exterior product, 
and  $ \mathbf c ( \cdot )$ denotes the Clifford action. 
On the two-point space, 
let
\begin{align}
\label{eq:rhoHS}
\rho_{ \mathrm{HS} } : M_{k} (\mathbb{C}) \to M_{k} (\mathbb{C}) 
\end{align}
be the orthogonal projection, with regard to the Hilbert-Schmidt scalar product,
to the subspace of scalar matrices. 
We will use the same notation for different  $k$ if no confusion arises.
Then
\begin{align*}
\sigma_2 \circ m:
B ( \mathcal H_2)  \otimes B ( \mathcal H_2) \to B ( \mathcal H_2)
,  \, \,
Q_1 \otimes Q_2 \to (1- \rho_{ \mathrm{HS} }) (Q_1 Q_2)
\end{align*}

\begin{lemma} \label{lem:sig-m}
The image of the junk forms
$ \widehat{ \pi}_D \brac{ \du \brac{ J_0 \cap \Omega^{1}_u (\mathcal A)} }$
coincides with $ \pi ( \mathcal A)$, consisting of diagonal matrices: 
\begin{align*}
\set{
\begin{bmatrix}
   f_1 & 0 \\  0  & f_2
\end{bmatrix}    
,
\,\,
f_1 , f_2 \in \mathcal{A}_1
}.
\end{align*}
Therefore  $\sigma_2 \circ m : 
\Omega^1_D (\mathcal A) \otimes \Omega^1_D (\mathcal A)
\to B ( \mathcal H) $
can be described as follows.
\begin{enumerate}[$1)$]
\item If $ u \otimes v \in \mathcal E_1 \otimes \mathcal E_2$, or
$ u \otimes v \in \mathcal E_2 \otimes \mathcal E_1 $, we have 
$\sigma_2 ( uv ) = uv $. 

\item 
If $ u \otimes v \in \mathcal E_1 \otimes \mathcal E_1$, 
say $u = u_1 \otimes u_2$ and $v = v_1 \otimes v_2$, where 
$u_1 , v_1 \in \Lambda^1(M)$ are one-forms on $M$,
and $ u_2 , v_2 \in \mathcal A_2$ are functions on the two-point space,
then
\begin{align}
\label{eq:s2-m-1}
\sigma_2 ( u v) 
= \mathbf c (u_1 \wedge  v_1) \otimes  ( u_2 v_2).
\end{align}

\item If $u \otimes v \in \mathcal E_2 \otimes \mathcal E_2$,  
and $u = \gamma u_1 \otimes u_2$ and $v = \gamma v_1 \otimes v_2$, with 
$u_1, v_2 \in \mathcal A_1$ are functions on $M$, 
and  $ u_2 , v_2 \in  \Omega^1_{ D_\phi} ( \mathcal A_2)$
are one-forms on the two-point space, 
\begin{align}
\label{eq:s2-m-2}
\sigma_2 ( u v) = u_1 v_1 
\otimes (1- \rho_{ \mathrm{HS}} ) \brac{ u_2 v_2 }.
\end{align}
\end{enumerate}
\end{lemma}
\begin{proof}
We refer to Lemma 6 and 7 in \cite[Ch. 6, Sect. 3]{Con-94} for details. 
\end{proof}

\begin{prop} \label{prop:Tsig}
Consider the product-type connection $\nabla$ given in \eqref{eq:nab-prdty},
its algebraic torsion
$T^\nabla_{ \sigma_2 }: \Omega^1_D ( \mathcal A)  \to \Lambda^2_D ( \mathcal A) $ 
is computed as follows:
\begin{align}
\label{eq:Tsig-prop}
T^\nabla_{ \sigma_2 } (w)
=
\begin{cases}
0 & w \in \mathcal E_1,
\\
\sigma_2 \brac{  w D_2 } & w \in \mathcal E_2
.
\end{cases} 
\end{align}
\end{prop}
\begin{proof}
As $ T^\nabla_{ \sigma_2}$ is left  $\mathcal A$-linear, 
it suffices, for  $w \in \mathcal E_1$,  to prove the special
case in which $w = d f \otimes 1$ for some $f \in \mathcal A_1$.
Also for $ w \in \mathcal E_2$, 
we can assume that $ w = \gamma \otimes \eta$, where  $\eta$ is 
the one-form of the two-point space defined in \eqref{eq:beta1}.

When $ w = d f \otimes 1$,  $ \nabla w \in \mathcal E_1 \otimes \mathcal E_1$
is given by \eqref{eq:N-1} with the second vanishes,
so that \eqref{eq:s2-m-1} holds, together, we obtain:
\begin{align*}
T^\nabla_{ \sigma_2} (w) 
=
\mathbf c \brac{ \wedge ( \nabla^{(1)} df) } \otimes 1 =0,
\end{align*}
where $ \nabla^{(1)}$ is the Levi-Civita connection on the manifold $M$. 
The torsion-free property  implies that  $ \nabla^{(1)} df $
is a symmetric $2$-tensor belong to the kernel of 
the exterior multiplication 
$ \wedge: \Lambda^1 (M) \otimes \Lambda^1 (M) \to \Lambda^2 (M)$.

For part 2), we need the calculation in \SS\ref{sec:2pt}. 
More precisely, we recall from \eqref{eq:nab-Z2} that 
$ \nabla^{(2)} \eta = c \eta \otimes \eta  $ where
$ c = \diag(1,-1)$. 
Now take $  w = \gamma \otimes \eta$, 
it remains to show that $m ( \nabla w) = w D_2$.
Indeed,  we apply \eqref{eq:N-2} with the first term vanishes: 
\begin{align}
 m \brac{  \nabla w } 
&=
m \brac{
\brac{ \gamma \otimes c \eta } \otimes \brac{ \gamma \otimes \eta }
}
=
\gamma^2 \otimes c \eta^2  
=
\gamma^2 \otimes  \eta D_\phi  
\nonumber \\
&=
\brac{ \gamma \otimes \eta }
\brac{ \gamma \otimes D_\phi }
= w D_2
,
\label{eq:m-nab-w}
\end{align}
where $ c \eta^2 = \eta D_\phi$ was obtained before in  \eqref{eq:T-beta}. 
\end{proof}

Since arbitrary $w \in \mathcal E_2$ reads:
\begin{align*}
w =  
\begin{bmatrix} 
0 &   \gamma  f^+ \otimes \phi\\
\gamma f^- \otimes \phi^*  & 0
\end{bmatrix} 
,
\end{align*}
the right hand side of \eqref{eq:Tsig-prop} can be explicitly computed 
using \eqref{eq:s2-m-1}: 
\begin{align}
\nonumber
\sigma_2 ( w D_2)
& =  
\sigma_2 \brac{ 
\begin{bmatrix} 
0 &   \gamma  f^+ \otimes \phi\\
\gamma f^- \otimes \phi^*  & 0
\end{bmatrix} 
\begin{bmatrix} 
0 &   \gamma   \otimes \phi\\
\gamma  \otimes \phi^*  & 0
\end{bmatrix} 
}
\\
&= 
\begin{bmatrix} 
f^+ \otimes (1- \rho_{ \mathrm{HS}}) \brac{  \phi \phi^*}  &  0 \\
0 & f^- \otimes (1- \rho_{ \mathrm{HS}}) \brac{\phi^* \phi   }
\end{bmatrix} 
.
\label{eq:sig-wD}
\end{align}
As a consequence,  if one of $ \phi \phi^*$ 
and  $ \phi^* \phi$, is a scalar matrix, so is the other, 
then $ ( 1- \rho_{ \mathrm{HS}}) ( \phi \phi^*) $
and $ ( 1- \rho_{ \mathrm{HS}}) ( \phi \phi^*)$ are both zero. Then $T^\nabla_{ \sigma_2 } (w) = 0$ for all $w\in \Omega_D^1$.
This is certainly the case for \cite[\SS4.2]{DSZ24} in which $ \phi \in \mathbb{C}$. 

We see that, in contrast with the spectral torsion,
junk forms from the manifold kill
the torsion generated by the connection on the two-point space in this setting.
Our solution to improve on this  discrepancy is to work with another differential
calculus following \cite{MR24}.

\subsection{Algebraic Torsion in the Mesland-Rennie Construction}
Parallel to Proposition \ref{prop:Tsig}, we have
\begin{prop} \label{prop:TY}
Let $ \nabla $ be the product-type connection in \eqref{eq:nab-prdty}, 
we have
\begin{enumerate}[$1)$]
\item 
For $ w \in \mathcal E_1$, $ T^\nabla_{ \Psi} (w) = 0$$;$ 
\item
For $ w \in \mathcal E_2$, $ m \brac{  T^\nabla_{ \Psi} (w) }  = w \mathcal D_2 $.
\end{enumerate}
\end{prop}
\begin{proof}
We can assume $ w = df \otimes 1$ for some $f \in \mathcal A_1$ for part 1), 
and for part 2),  $ w= \gamma \otimes \eta$, where $ \eta$ is the one-form defined 
\eqref{eq:beta1}. 
The general case follows from the left  $\mathcal A$-linearity of $ T^\nabla_{
\Psi}$.

Let us take $ w = df \otimes 1$, as the Levi-Civita connection $\nabla^{(1)}$
is torsion-free, we have
\begin{align}
\label{eq:LC-sym}
\brac{ \nabla^{(1)} df }_{ (0)} 
\otimes
\brac{ \nabla^{(1)} df }_{ (1)}
=
\brac{ \nabla^{(1)} df }_{ (1)} 
\otimes
\brac{ \nabla^{(1)} df }_{ (0)} .
\end{align}
is a symmetric tensor.
According to \eqref{eq:N-1}, 
$ \nabla w = \nabla ( \tilde w \otimes 1) \in \mathcal E_{ (1,1)}$
is determined by $ \nabla^{(1)} df $ and is, in particular, symmetric,
meaning that it belongs to the image of $\Psi$ (given in \eqref{eq:Y-F1}).
In other words,
$  (1 - \Psi ) \brac{ \nabla w }  = 0$, which proves the first claim.

For part 2), we set $ w = \gamma \otimes \eta $ and keep the notations 
as in \eqref{eq:m-nab-w}.
We have seen  that 
$ \nabla w \in \mathcal E_2 \otimes_{ \mathcal A} \mathcal E_2$, 
on which $ \Psi = 0$ (cf. \eqref{eq:Y-F1}), that is 
\begin{align*}
 T^\nabla_{ \Psi} (w) = (1-\Psi) ( \nabla w) = \nabla w,
\end{align*}
and then \eqref{eq:m-nab-w} concludes the proof:
\begin{align*}
 m \brac{  T^\nabla_{ \Psi} (w) } 
 = m ( \nabla w) = w \mathcal D_2.
\end{align*}
\end{proof}

We thus see that in this approach the torsion fits better (on $\mathcal E_2$)
with the spectral torsion, 
but we need consider some other connections for further improvement.

\subsection{Recovering The Spectral Torsion Functional}
 
Now our objective is to look for another connection whose 
algebraic torsion functionals defined in \eqref{eq:Tsig} or \eqref{eq:TPsi}
recovers the spectral one computed in Theorem \ref{thm:SpT-MZ}. 
Equivalently, we would like to reproduce the left $ \mathcal A$-module 
map $ T_D : \Omega^1_D (\mathcal A) \to B( \mathcal H)$ defined in \eqref{eq:T_D}.

In Theorems \ref{prop:Tsig} and \ref{prop:TY} we have seen that, for the product-type connection $\nabla$
\begin{align*}
T^\nabla_{ \sigma_2} (w) = m\brac{ T^\nabla_{ \Psi} (w)} = 0 
,\, \, 
\forall w \in \mathcal E_1.
\end{align*}
By comparison with $T_D$, we need thus to perturb the connection $\nabla$
by adding the following left $ \mathcal A$-module map
$S : \Omega^1_D (\mathcal A) \to T^2_D(\mathcal A)$:
\begin{align}
\label{eq:S}
S (w) = 
\begin{cases}
(1 - \Psi) ( w \otimes \mathcal D_2), & w \in \mathcal E_1, 
\\ 
0, & w \in \mathcal E_2.
\end{cases}
\end{align}
\begin{lemma} 
\label{lem:m-S}
The left $ \mathcal A$-module map $S$ above is designed in such a way that
\begin{align}
\label{eq:m-S}
m \circ S (w) = 
\begin{cases}
  w \mathcal D_2 = T_D (w), & w \in \mathcal E_1, 
\\ 
0, & w \in \mathcal E_2.
\end{cases}
\end{align}
\end{lemma}
\begin{proof}
For $w \in \mathcal E_1$, 
$ w \otimes \mathcal D_2 \in \mathcal E_1 \otimes \mathcal E_2$
so that $ \Psi $ is defined by \eqref{eq:Y-F2}. 
In particular, 
$ (1-\Psi) (w \otimes \mathcal D_2)
= \frac{1}{2} \brac{ 1 - \beta_{ 12} } (w \otimes \mathcal D_2)$.
To conclude the proof, we just need to show that 
$m \circ \beta_{ 12}  ( w \otimes \mathcal D_2) = - w \mathcal D_2 $. In fact,  
for $w = w_1 \otimes w_2 \in \mathcal E_1 = 
\Omega^1_{ D_M} ( \mathcal A_1) \otimes_{ \mathcal A} \mathcal A_2$ 
and compute:
\begin{align*}
m \circ \beta_{ 12}  ( w \otimes \mathcal D_2) 
&= 
m \brac{  \mathcal D_2 \otimes \alpha_1 (w) } 
= \gamma w_1 \otimes D_\phi \alpha (w_2)
\\
&=
- w_1 \gamma \otimes w_2 D_\phi 
=
- w \mathcal D_2,
\end{align*}
where the crucial property is the fact that the Clifford action of the one-form
$w_1$ anti-commutes with the grading operator $ \gamma $. 
\end{proof}

As for the algebraic torsion functional \eqref{eq:Tsig-prop} in Connes' calculus,    
there is also a discrepancy caused by the projection $\sigma_2$ or, equivalently, $\rho_{ \mathrm{HS}}$
in \eqref{eq:rhoHS} and the agreement with the spectral one occurs on a smaller domain of 
$ u, v, w \in \Omega^1_D ( \mathcal A)$
such that $\sigma_2(uv)w=uvw$:
\begin{thm}
\label{thm:TvsCn}
For the non-product type connection 
$ \widetilde \nabla = \nabla + S$
where $S$ is defined in \eqref{eq:S},
we have $ T^{ \widetilde \nabla }_{ \sigma_2} = \sigma_2 \circ T_D$.
In particular, the algebraic torsional functional defined in \eqref{eq:Tsig-prop}
agrees with $ \widetilde{\mathscr T}_D$: 
\begin{align*}
\mathscr T_{ \sigma_2} ( u , v ,w)
=  
\int^{\nres} u v T^{ \widetilde \nabla}_{ \sigma_2 } (w)
=
\widetilde{\mathscr T}_D  ( u , v, w), 
\end{align*}
where  $ u, v, w \in \Omega^1_D ( \mathcal A)$ and  
$ \widetilde{\mathscr T}_D$ 
is a reduced version of the spectral torsion functional 
${\mathscr T}_D$,
\begin{align*}
\widetilde{\mathscr T}_D  ( u , v, w) 
= \int^{\nres} \sigma_2 ( u v) w \abs D^{-m} 
=
\int^{\nres} \sigma_2 ( u v) T_D (w)
.
\end{align*}
\end{thm}
\begin{proof}
Recall from Proposition \ref{prop:tor-purt}: 
$ T_\sigma^{ \widetilde \nabla} = T_\sigma^{ \nabla} + S_\sigma $, 
with $S_\sigma = \sigma_2 \circ m \circ S$. 
The equality $ $
$ T^{ \widetilde \nabla }_{ \sigma_2} = \sigma_2 \circ T_D$ is achieved 
by design, it is a straightforward consequence of Lemma \ref{lem:m-S}
and Proposition \ref{prop:Tsig}.

Given  one-forms $ u, v, w \in \Omega^1_D ( \mathcal A)$, 
we have  $ u v$, $T^{ \widetilde \nabla}_{ \sigma_2 } (w)$ and  $ T_D (w)$ 
all belong to the image of 
$ \widehat{ \pi}_D ( \Omega_u^2 ( \mathcal A))$, thus 
admit the orthonormal decomposition  as in \eqref{eq:sig-k} (with $k=2$).
Firstly, one has to slightly adjust the proof of Theorem \ref{thm:SpT-MZ}
to conclude that 
\begin{align*}
\widetilde{\mathscr T}_D  ( u , v, w) 
= 
\int^{\nres} \sigma_2 ( u v) w \abs D^{-m} 
=
\int^{\nres} \sigma_2 ( u v) T_D (w).
\end{align*}
The decomposition \eqref{eq:sig-k} gives 
\begin{align*}
\int^{\nres} \sigma_2 ( u v) T_D (w)
=
\int^{\nres}  u v\sigma_2 (T_D (w)) =
\int^{\nres}  uv T^{ \widetilde \nabla}_{ \sigma_2} (w)   
.
\end{align*}
The proof is complete.
\end{proof}

We now have arrived at the highlight of the paper. 
For the almost noncommutative manifold $M \otimes \Z_2$,
We have found an appropriate differential calculus 
(the construction of the projection $\Psi$ in the Mesland-Rennie approach),
and a connection whose algebraic torsion agrees with the spectral one,
which is intrinsic to the spectral date. 

\begin{thm}\label{2nd main}
Let $ \widetilde \nabla = \nabla + S$ be the non-product type connection
as before.
Then $ T^{ \widetilde \nabla}_\Psi = T_D$, in other words,
its algebraic torsion functional 
$ \mathscr T_\Psi $ defined in \eqref{eq:TPsi} for $\widetilde \nabla$
recovers the spectral torsion functional 
$ \mathscr T^D$ in Theorem \ref{thm:SpT-MZ}.
\end{thm}
\begin{proof}
According to Prop. \ref{prop:tor-purt}, 
$ T^{ \widetilde \nabla}_\Psi = T_\Psi^\nabla + S_{ \Psi} $ with 
$ S_\Psi =m \circ (1 -\Psi) \circ S$. 
Since $ (1 - \Psi)^2 = 1- \Psi$, we have $ m \circ S_\Psi = m \circ S$. 
Therefore  the equality $ m \circ T^{ \widetilde \nabla}_\Psi = T_D$
follows immediately from Lemma \ref{lem:m-S} and Proposition \ref{prop:TY}.
As a result, the associated trilinear functionals \eqref{eq:SpT-MZ}
and \eqref{eq:TPsi} are identical as well. 
The proof is complete. 
\end{proof}
\begin{rmk}\label{rmk:pertMz2}
Of course, the connection $\widetilde \nabla$ can be also obtained as
a perturbation of the product of the Levi-Civita connection on $M$ with the
Grassmann torsion free connection on $\mathcal{Z}_2$ 
by perturbing first the latter one according to Remark\,\ref{rmk:pertz2} and then adding the $S$ as above.
\end{rmk}

\section{Final Comments}
We have shown that for the simplest quantum geometry of  
$\Z_2$ there is a unique connection of which the (algebraic) torsion functional is equal to the spectral torsion functional.
Instead for the general almost commutative geometry on $M\times \Z_2$ in \cite{DSZ24} the torsion of a linear connection for the Connes calculus can reproduce at most the reduced spectral torsion functional,
while for the Mesland-Rennie calculus there is a non-product type connection of which the algebraic torsion exactly equals the (full) spectral torsion functional.
We also extended these results to the case when the parameter $\phi$ of the internal Dirac operator is not a complex scalar.
Clearly more examples should be studied and then more general relations established between the spectral and algebraic torsion.

\end{document}